\documentclass[a4paper,11pt]{amsart}
\usepackage[english]{babel}
\usepackage{amstext,amsfonts,amsthm,indentfirst,graphicx,amssymb,amscd,epsfig}
\usepackage{amsmath}
\usepackage{graphics,wrapfig}
\usepackage[ansinew]{inputenc}    
\usepackage{bm}
\usepackage{mathrsfs,dsfont}
\usepackage{hyperref}

\theoremstyle{definition}
\newtheorem{definition}{Definition}[section]
\newtheorem{ex}[definition]{Example}
\newtheorem{rem}[definition]{Remark}

\theoremstyle{plain}
\newtheorem{prop}[definition]{Proposition}
\newtheorem{lem}[definition]{Lemma}
\newtheorem{coro}[definition]{Corollary}
\newtheorem{teo}[definition]{Theorem}

\newfont{\bbb}{msbm10 scaled\magstephalf}     
\def\im{\mathbf{\hat{i}}}
\def\jm{\mathbf{\hat{j}}}
\def\km{\mathbf{\hat{k}}}

\def\X{\mathds X}

\def\R{\mathbb R}

\def\R{\mbox{\bbb R}}
\def\x{\mathbf{x}}

\def\eu{\mathbf{e}_1}
\def\ed{\mathbf{e}_2}
\def\EO{E_\Omega}
\def\FO{F_\Omega}
\def\GO{G_\Omega}
\def\eo{e_\Omega}
\def\fuo{f_{1\Omega}}
\def\fdo{f_{2\Omega}}
\def\go{g_\Omega}
\def\n{\mathbf{n}}
\def\p{\mathbf{p}}
\def\I{\mathbf{I}}
\def\Ib{\bar{\mathbf{I}}}
\def\II{\mathbf{II}}

\def\w{\mathbf{w}}
\def\id{\mathbb{I}}
\def\Wm{\mathds{W}}

\def\P{\mathds{P}}
\def\Q{\mathds{Q}}
\def\Omegam{\mathbf{\Omega}}
\def\Omegamb{\bar{\mathbf{\Omega}}}
\def\Omegab{\bar{\Omega}}
\def\Lambdam{\mathbf{\Lambda}}
\def\Lambdamb{\bar{\mathbf{\Lambda}}}

\def\mum{\boldsymbol{\mu}}
\def\num{\boldsymbol{\nu}}
\def\rhom{\boldsymbol{\rho}}
\newcommand{\spann}[2]{\left\langle{#1},{#2}\right\rangle}
\newcommand{\Ga}[3]{\Gamma_{{#1}{#3}}^{#2}}
\newcommand{\Ta}[3]{\mathcal{T}_{{#1}{#3}}^{#2}}
\newcommand{\la}[2]{\lambda_{{#1}{#2}}}
\newcommand{\laO}{\lambda_{\Omega}}
\newcommand{\KO}{K_{\Omega}}
\newcommand{\HO}{H_{\Omega}}
\newcommand{\TA}[1]{\boldsymbol{\mathcal{T}}_{#1}}
\newcommand{\GA}[1]{\mathbf{\Gamma}_{#1}}
\newcommand{\TAb}[1]{\bar{\boldsymbol{\mathcal{T}}}_{#1}}
\newcommand{\TAh}[1]{\hat{\boldsymbol{\mathcal{T}}}_{#1}}
\newcommand{\GAb}[1]{\bar{\mathbf{\Gamma}}_{#1}}

\title[The fundamental theorem for singular surfaces]{The fundamental theorem for singular surfaces with limiting tangent planes}

\author{T. A. Medina-Tejeda}

\date{}

\address{Instituto de Ci\^encias Matem\'aticas e de Computa\c{c}\~ao - Universidade de S\~ao Paulo,
Av. Trabalhador s\~ao-carlense, 400 - Centro,
CEP: 13566-590 - S\~ao Carlos - SP, Brazil}

\email{tamedinat@usp.br}

\thanks{The author is supported by CAPES Grant no. PROEX-10359340/D}

\begin{document}
\begin{abstract}
In this paper, we prove a similar result to the fundamental theorem of regular surfaces in classical differential geometry, which extends the classical theorem to the entire class of singular surfaces in Euclidean 3-space known as frontals. Also, we characterize in a simple way these singular surfaces and its fundamental forms with local properties in the differential of its parametrization and decompositions in the matrices associated to the fundamental forms. In particular we introduce new types of curvatures which can be used to characterize wave fronts. The only restriction on the parametrizations that is assumed in several occasions is that the singular set has empty interior. 
\end{abstract}

\maketitle

\section{Introduction}

In recent years, there is a great interest  in the geometry of a special type of singular surface, namely, frontal. Many papers are dedicated to the study of frontals from singularity theory and geometry viewpoints \cite{sajicuspidal,ishifrontal2,ishifrontal}, in particular  wave fronts a subclass of these \cite{arnld-sing-caus,kos,front,MSUY}. The word "front" comes from physical fronts, bounding a domain in which a physical process propagates at a fixed moment in time. For instance, a wave propagating in the 3-Euclidean space with constant speed starting from each point of an ellipsoid in direction of the interior of this (the initial domain to be perturbed)  creates a equidistant surface at time t bounding an interior part of the ellipsoid that it has not been perturbed at time t. In this case, the complete equidistant surface is called
the wave front, this changes as time passes leading to the formation of singularities
along the whole equidistant surface in any time \cite{arnld-sing-caus}. The notion of "frontal" surged as a natural generalization of wave front in the case of hypersurfaces and a generalized definition with equivalences can be found in \cite{ishifrontal2}. A smooth map $\x: U \to \R^3$ defined in an open set $U \subset \R^2$ is called a {\it frontal} if, for all $\p\in U$ there exists a unit normal vector field $\num:V_p \to \R^3$ along $\x$, where $V_p$ is an open set of $U$, $\p\in V_p$. This means, $|\num|=1$ and it is orthogonal to the partial derivatives of $\x$ for each point $(u,v)\in V_p$. If also the singular set $\Sigma(\x)=\{\p\in U: \x \text{ is not immersive at $\p$} \}$ has empty interior we call $\x$ a {\it proper frontal}. Since $\Sigma(\x)$ is closed, this is equivalent to have $\Sigma(\x)^c$ being dense and open in $U$. A frontal $\x$ is a {\it wave front} or simply {\it front} if the pair $(\x,\num)$ is an immersion for all $\p\in U$. There are many examples of frontals which are not wave fronts, cuspidal $S_k$ singularities for instance \cite{sajicuspidal}. The existence of a smooth normal vector field on these singular surfaces determines a plane (the orthogonal space) at singular points that can be understood as a limiting plane of the tangent planes on regular points around them (see Figure \ref{cusp}).

\begin{figure}[h]
	
	\begin{center}
		\includegraphics[scale=0.21]{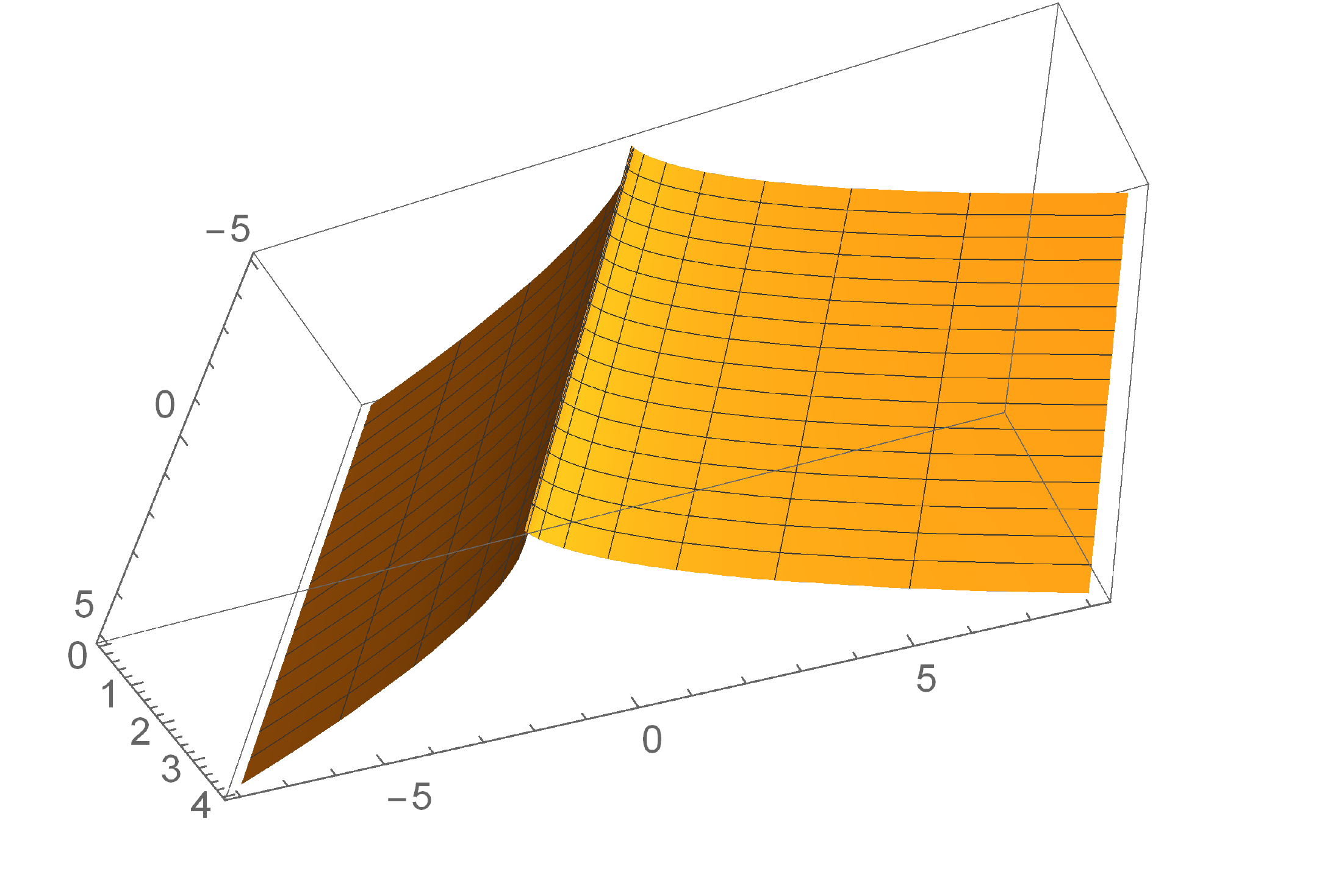}\qquad\qquad \includegraphics[scale=0.29]{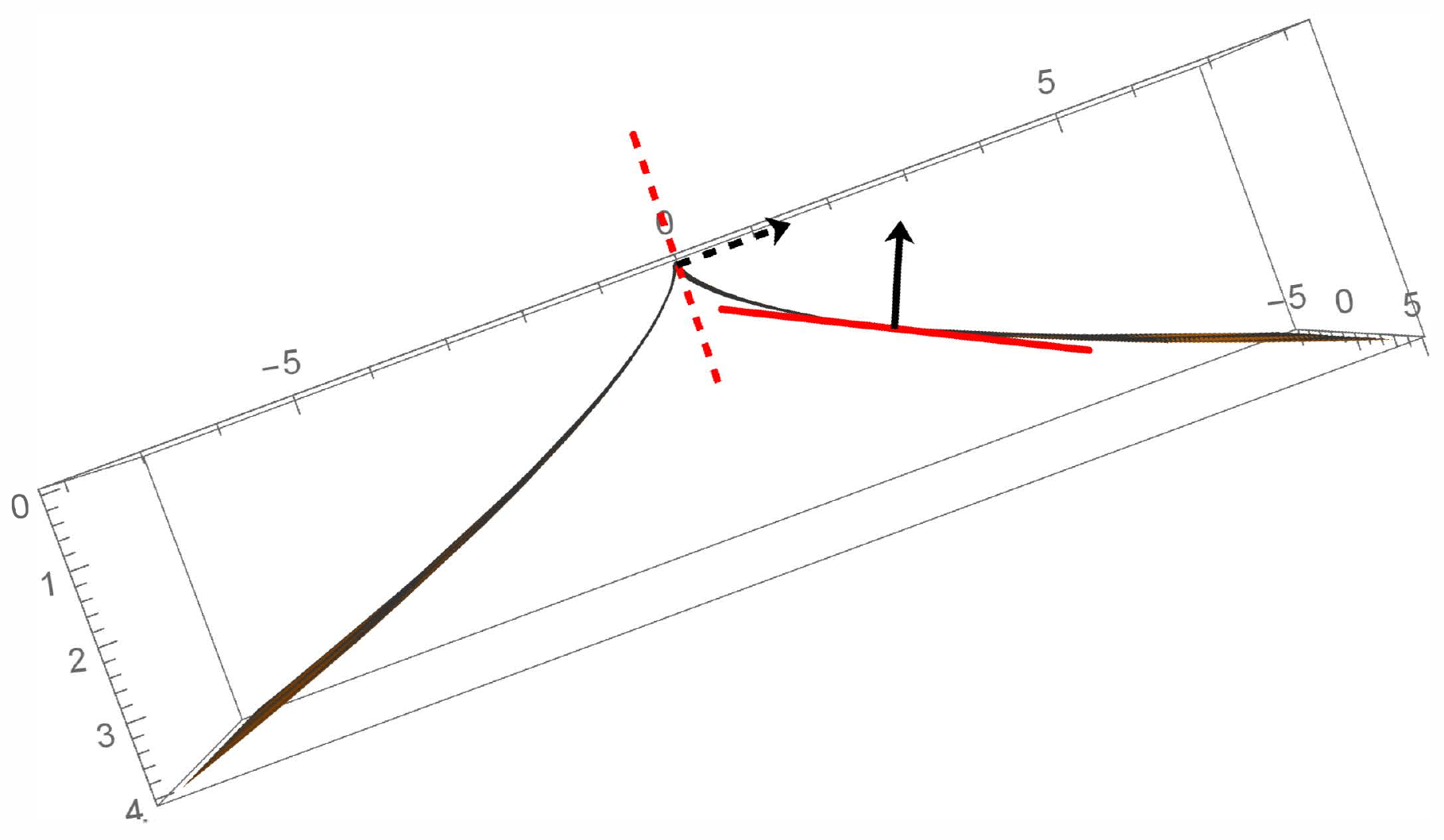}
	\end{center}
	\caption{The cuspidal edge $(\x(u,v)=(u, v^2, v^3))$ and the limiting tangent planes.}\label{cusp}
	
\end{figure}
The cuspidal edge and the swallowtail (see Figure \ref{cusp} and \ref{swallow}) are two types of singular points that represent
the generic singularities in the space of wave fronts with the Whitney $C^{\infty}$-topology. For this reason, all the re-parametrizations and diffeomorphic singular surfaces to these are the most studied and there exist criterias to recognize them \cite{KRSUY,ishifrontal}. However, these singularities are not generic in the space of all frontals (in fact proper frontals are not generic either)\cite{ishifrontal2}. There are some non-proper frontals which are not "surfaces", $\x(u,v)=(uv,0,0)$ for instance and others whose entire image is a surface but locally at some singular points the image of a neighborhood at these is a constant (see example 2.5 \cite{ishifrontal2}). Here we treat frontals in general, but our main result aim to proper frontals. 
\begin{figure}[h]
	
	\begin{center}
		\includegraphics[scale=0.25]{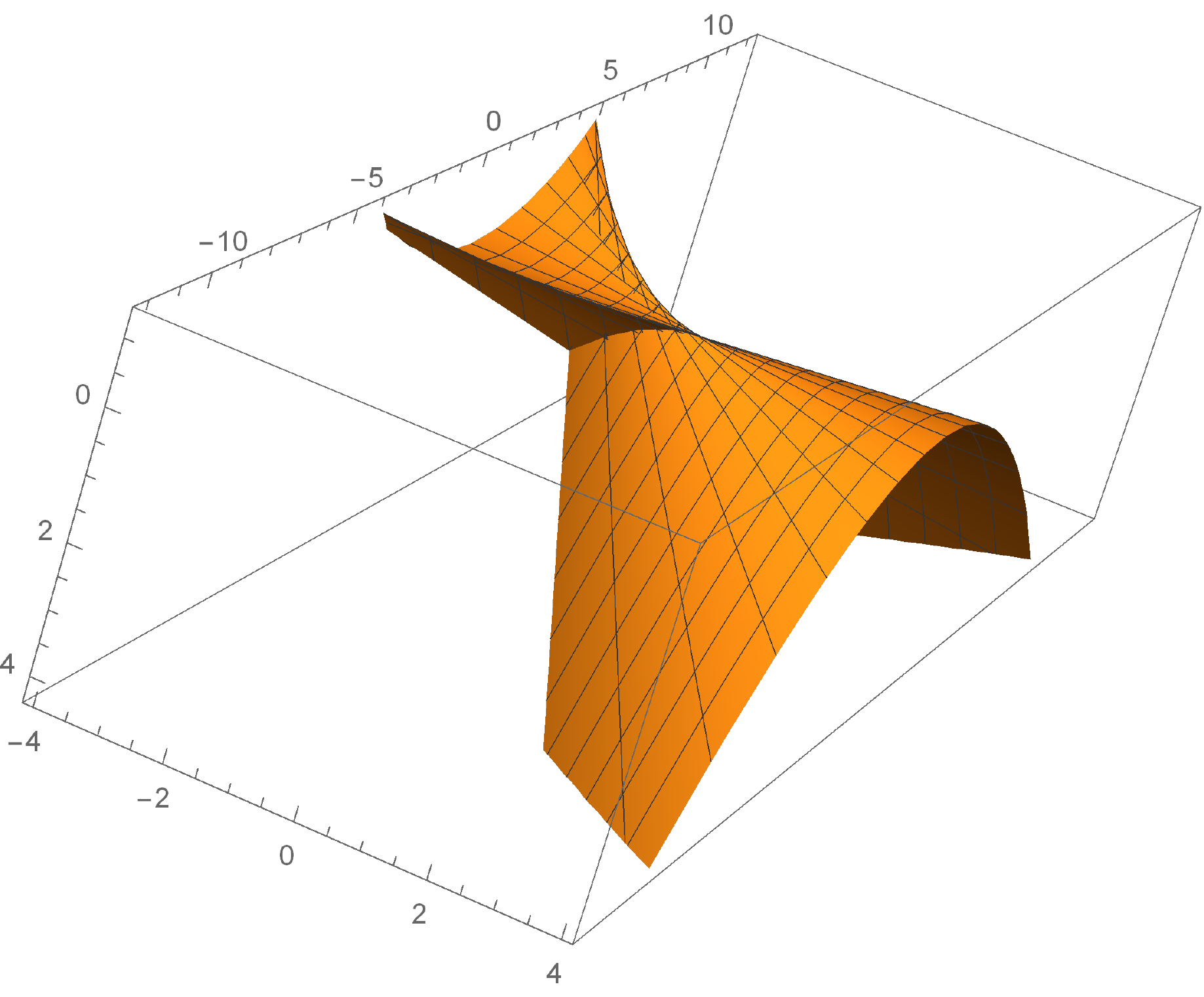}\qquad\qquad \includegraphics[scale=0.35]{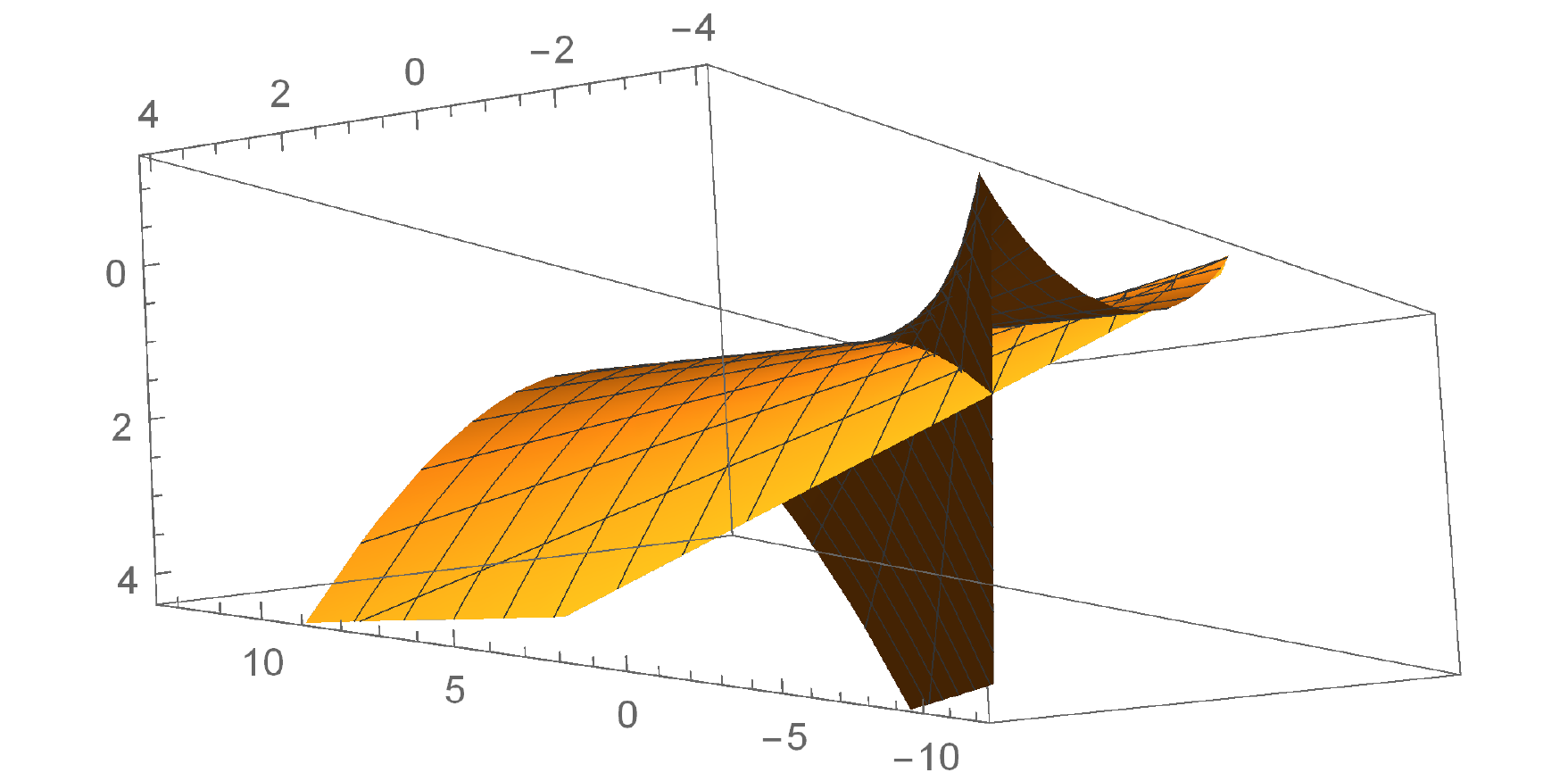}
	\end{center}
	\caption{The swallowtail $(\x(u,v)=(3u^4 + u^2v, 4u^3 + 2uv, v))$, an example of front.}\label{swallow}
	
\end{figure}

 In classical differential geometry, the fundamental theorem of regular surfaces (see\cite{dc,jjs}) states that if we have $E,F,G,e,f,g$ smooth functions defined in an open set $U\subset \R^2$, with $E>0$, $G>0$, $EG-F^2>0$ and the given functions satisfy formally the Gauss and Mainardi-Codazzi equations, then for each $\p\in U$ there exist a neighborhood $V\subset U$ of $p$ and a diffeomorphism $\x:V \to \x(V)\subset \R^3$ such that the regular surface $\x(U)$ has $E,F,G$ and $e,f,g$ as coefficients of the first and second fundamental forms, respectively. Furthermore, if $U$ is connected and if ${\mathbf{\bar{\x}}}:U \to {\mathbf{\bar{\x}}}(U)\subset \R^3$ is another diffeomorphism satisfying the same conditions, then there exist a translation $\mathbf{T}$ and a proper linear orthogonal transformation $\rhom$ in $\R^3$ such that ${\mathbf{\bar{\x}}}=\mathbf{T}\circ \rhom \circ \x$.\\Gauss equation: 
\begin{align}
\Ga{1}{2}{2u}-\Ga{1}{2}{1v}+\Ga{1}{1}{2}\Ga{1}{2}{1}+\Ga{1}{2}{2}\Ga{1}{2}{2}-\Ga{1}{2}{1}\Ga{2}{2}{2}-\Ga{1}{1}{1}\Ga{1}{2}{2}=-EK\label{Gauss}
\end{align}	
Mainardi-Codazzi equations:
\begin{subequations}
	\label{MC}
	\begin{align}
	e_v-f_u=e\Ga{1}{1}{2}+f(\Ga{1}{2}{2}-\Ga{1}{1}{1})-g\Ga{1}{2}{1}\label{MC1}\\
	f_v-g_u=e\Ga{2}{1}{2}+f(\Ga{2}{2}{2}-\Ga{1}{1}{2})-g\Ga{1}{2}{2}\label{MC2}
	\end{align}	
\end{subequations}
where $K$ is the Gaussian Curvature and $\Ga{i}{j}{k}
$ the Christoffel symbols.

This theorem realizes first and a second fundamental forms compatibles as a regular surface in the euclidean 3-space. In \cite{kos} M. Kossowski gave sufficient conditions for a singular first fundamental form to be realized as a wave front with several restricted characteristics. In Section 5, we prove our main result in theorem \ref{TF} which generalizes the fundamental theorem of regular surfaces mentioned before including now all the proper frontals, with the possibility to distinguish wave fronts from its fundamental forms. To state this theorem, we introduced some additional terminology in Section 2, where we establish the necessary notation, terminology and basic results that we use mostly. In Section 3 we characterize a frontal $\x$ with the differential of $\x$, its fundamental forms with decomposition of matrices and wave fronts with two new curvatures which are related with the Gaussian and mean curvature. In Section 4, we get two groups of equations, which are present in all frontals and guarantee the integrability conditions for the system of PDE that we consider in theorem \ref{TF}. After finishing this paper, I was informed by professor Takashi Nishimura about recent papers by T. Fukunaga and M. Takahashi on geometry of frontals. In \cite{fuku}, they use orthonormal moving frames to study basic invariants and curvatures of framed surfaces. As in our corollary \ref{wf}, they also characterized wave fronts in terms of curvatures.       
\section{Fixing notation, definitions and basic results}\label{section-definition}
We denote $U$ and $V$ in this paper open sets in $\R^2$. Let $\x:U \to \R^3$ be a frontal, and as we are interested in exploring local properties of frontals, restricting the domain if necessary, we can suppose that we have a global normal vector field $\num:U \to \R^3$. There are two possible choices of normal vector fields along $\x$ ($\num$ and $-\num$). We are always assuming that we have chosen one of them and we hold fixed this for all the concepts defined using a normal vector field along $\x$. Let $\mathbf{f}:U \to \R^n$ be a smooth map, we denote by $D\mathbf{f}:=(\frac{\partial \mathbf{f}_i}{\partial x_j})$, the differential of $\mathbf{f}$ and we consider it as a smooth map $D\mathbf{f}:U \to \mathcal{M}_{n\times2}(\R)$. We write $D\mathbf{f}_{x_1}$, $D\mathbf{f}_{x_2}$ the partial derivatives of $D\mathbf{f}$ and $D\mathbf{f}(\p):=(\frac{\partial \mathbf{f}_i}{\partial x_j}(\p))$ for $\p \in U$. Also, all vector in $\R^n$ is identified as vector column in $\mathcal{M}_{n\times1}(\R)$ and if $\mathbf{A}\in \mathcal{M}_{n\times n}(\R)$, $\mathbf{A}_{(i)}$ is the $i^{th}$-row and $\mathbf{A}^{(j)}$ is the $j^{th}$-column of $\mathbf{A}$.  

\begin{definition}
	We call {\it moving base} a smooth map $\Omegam:U \to \mathcal{M}_{3\times2}(\R)$ in which the columns $\w_1, \w_2:U \to \R^3$ of the matrix $\Omegam=\begin{pmatrix}\w_1&\w_2\end{pmatrix}$ are linearly independent smooth vector fields.    	
\end{definition}

\begin{definition}
We call a {\it tangent moving base} of $\x$ a moving base $\Omegam=\begin{pmatrix}\w_1&\w_2\end{pmatrix}$ such that $\x_u,\x_v \in \spann{\w_1}{\w_2}$, where $\spann{}{}$ denotes the linear span vector space.    	
\end{definition}
Let $\x:U \to \R^3$ be a frontal with a global normal vector field $\num:U \to \R^3$. Denoting the inner product by $(\cdot)$ and $()^T$ the operation of transposing a matrix, we set the matrices: 
\begin{subequations}
\label{FF}
\begin{align}
\mathbf{I}=
\begin{pmatrix}
E&F\\
F&G
\end{pmatrix}&:=\begin{pmatrix}
\x_u\cdot\x_u & \x_u\cdot\x_v\\
\x_u\cdot\x_v & \x_v\cdot\x_v
\end{pmatrix}\label{FF1}
\\
\mathbf{II}= 
\begin{pmatrix}
e&f\\
f&g
\end{pmatrix}&:=\begin{pmatrix}
\num\cdot\x_{uu} & \num\cdot\x_{uv}\\
\num\cdot\x_{uv} & \num\cdot\x_{vv}
\end{pmatrix}\label{FF2}	\\
\GA{1}=\begin{pmatrix}
\Ga{1}{1}{1} & \Ga{1}{2}{1}\\
\Ga{2}{1}{1} & \Ga{2}{2}{1}
\end{pmatrix}&:=\begin{pmatrix}
\frac{1}{2}E_u & (F_u-\frac{1}{2}E_v)\\
\frac{1}{2}E_v & \frac{1}{2}G_u
\end{pmatrix}\I^{-1}\label{GA1}
\\
\GA{2}=\begin{pmatrix}
\Ga{1}{1}{2} & \Ga{1}{2}{2}\\
\Ga{2}{1}{2} & \Ga{2}{2}{2}
\end{pmatrix}&:=\begin{pmatrix}
\frac{1}{2}E_v & \frac{1}{2}G_u\\
(F_v-\frac{1}{2}G_u) & \frac{1}{2}G_v
\end{pmatrix}\I^{-1}\label{GA2}
\\
\boldsymbol{\alpha}&:=-\II^{T}\I^{-1}
\end{align}
\end{subequations}
The matrices $\I$ and $\II$ in a non-singular point $\p\in U$ coincide with the matrix representation of the {\it first fundamental form} and of the {\it second fundamental form} respectively. $\GA{1}$, $\GA{2}$ and $\boldsymbol{\alpha}$ are defined in $\Sigma(\x)^c$, they are the Christoffel symbols and the Weingarten matrix. Also observe that, we can compute these matrices in this way: 
\begin{subequations}
	\label{X}
	\begin{align}\I&=D\x^TD\x\label{IX}
	\\
	\II&=-D\x^T D\num
	\\
	\GA{1}&=(D\x_u^TD\x)\I^{-1}\label{T1X} 
	\\
	\GA{2}&=(D\x_v^TD\x)\I^{-1}\label{T2X}\end{align}
\end{subequations}
Let $\Omegam=\begin{pmatrix}\w_1&\w_2\end{pmatrix}$ be a moving base, we denote by $\n:=\frac{\w_1\times\w_2}{\|\w_1\times\w_2\|}$ and we set the matrices:
\begin{subequations}
	\label{O}
	\begin{align}\I_\Omega&=\begin{pmatrix}
	\EO & \FO\\
	\FO & \GO\end{pmatrix}:=\Omegam^T\Omegam\label{IO}
	\\
	\II_\Omega&= 
	\begin{pmatrix}
	\eo&\fuo\\
	\fdo&\go
	\end{pmatrix}:=-\Omegam^T D\n\label{IIO}
	\\
	\TA{1}&=\begin{pmatrix}
	\Ta{1}{1}{1} & \Ta{1}{2}{1}\\
	\Ta{2}{1}{1} & \Ta{2}{2}{1}
	\end{pmatrix}:=(\Omegam_u^T\Omegam)\I_\Omega^{-1}\label{T1O} 
	\\
	\TA{2}&=\begin{pmatrix}
	\Ta{1}{1}{2} & \Ta{1}{2}{2}\\
	\Ta{2}{1}{2} & \Ta{2}{2}{2}
	\end{pmatrix}:=(\Omegam_v^T\Omegam)\I_\Omega^{-1}\label{T2O}
	\\
	\boldsymbol{\mu}&:=-\II_\Omega^T\I_\Omega^{-1}\label{W}
	\end{align}
\end{subequations}
Notice that, these last matrices coincide with $\I,\II,\GA{1},\GA{2}$ and $\boldsymbol{\alpha}$ when $\Omegam=D\x$ is a moving base. Since $\n\cdot\w_{1}=0$ and $\n\cdot\w_{2}=0$, then we have $-\n_u\cdot\w_{1}=\n\cdot\w_{1u}$, $-\n_v\cdot\w_{1}=\n\cdot\w_{1v}$,
$-\n_u\cdot\w_{2}=\n\cdot\w_{2u}$ and $-\n_v\cdot\w_{2}=\n\cdot\w_{2v}$. Therefore,
\begin{align}
\II_\Omega=\begin{pmatrix}
\n\cdot\w_{1u} & \n\cdot\w_{1v}\\
\n\cdot\w_{2u} & \n\cdot\w_{2v}
\end{pmatrix}\label{SFO} 
\end{align}
Also, as $\n_u,\n_v \in \spann{\w_{1}}{\w_{2}}$, there exist real functions  ($\bar{\mu}_{ij}$) $i,j\in\{1,2\}$ defined on $U$, such that: 
\begin{subequations}
	\begin{align}
	\n_u=\bar{\mu}_{11}\w_1+\bar{\mu}_{12}\w_2\\
	\n_v=\bar{\mu}_{21}\w_1+\bar{\mu}_{22}\w_2
	\end{align}
\end{subequations}
Then, $D\n=\Omegam\bar{\mum}^T$, where $\bar{\mum}=(\bar{\mu}_{ij})$. Thus, using (\ref{IIO}) $\II_\Omega=-\Omegam^T D\n=-\Omegam^T\Omegam\bar{\mum}^T=-\I_\Omega\bar{\mum}^T$, therefore $\bar{\mum}=-\II_\Omega^T\I_\Omega^{-1}=\mum$ and we have:
\begin{align}
D\n=\Omegam\mum^T\label{WO}
\end{align}
By last, $\w_1$ and $\w_2$ are linearly independent, the positive-definite quadratic form $(\cdot)$ restricted to $\spann{\w_1}{\w_2}$ has $\I_\Omega=\Omegam^T\Omegam$ as its matrix representation in the base $\{\w_1,\w_2\}$ and therefore $det(\I_\Omega)>0$.

The following is a particular version of Frobenius theorem that can be found in (appendix B\cite{jjs}) or \cite{clt}. 

\begin{teo}[Frobenius]\label{frobg}

Let $\mathbf{\Theta},\mathbf{\Xi}:U\times V\to\R^n$ be smooth vector fields, where $U\subset\R^2$ and $V\subset\R^n$ are open sets. Let $(u_0,v_0)\in U$ be a fixed point. Then for each point $\p \in V$ the system of PDE:
\begin{subequations}
\label{sys1}
\begin{align}
\frac{\partial \x}{\partial u}=\mathbf{\Theta} (u,v,\x(u,v)),\label{eq1.a}
\\
\frac{\partial \x}{\partial v}=\mathbf{\Xi} (u,v,\x(u,v)),\label{eq1.b}
\\
\x(u_0,v_0)=p,\label{eq1.c}
\end{align}	
\end{subequations}
has a unique smooth solution $\x:U_0 \to \R^n$ defined on a neighborhood $U_0$ of $(u_0,v_0) \in U_0$ if and only if, it satisfies the compatibility condition:

\begin{equation}
\frac{\partial \mathbf{\Theta}}{\partial v} + \frac{\partial\mathbf{\Theta}}{\partial\x}\mathbf{\Xi} =\frac{\partial \mathbf{\Xi}}{\partial u} + \frac{\partial\mathbf{\Xi}}{\partial\x}\mathbf{\Theta} \label{compatibility1}
\end{equation} 

\end{teo}

\begin{coro}\label{frobl}

Let $\mathbf{S},\mathbf{T}:U\to \mathcal{M}_{n\times n}(\R)$ be smooth vector fields, where $U$ is an open set in $\R^2$. Let $(u_0,v_0)\in U$ be a fixed point. Then for each point $\mathbf{A} \in GL(n)$ the system of PDE:
\begin{subequations}
\label{sys2}
\begin{align}
\frac{\partial \mathbf{G}}{\partial u}=\mathbf{S}\mathbf{G},\label{eql1.a}
\\
\frac{\partial \mathbf{G}}{\partial v}=\mathbf{T}\mathbf{G},\label{eql1.b}
\\
\mathbf{G}(u_0,v_0)=\mathbf{A},\label{eql1.c}
\end{align}	
\end{subequations}
has a unique smooth solution $\mathbf{G}:U_0 \to GL(n)$ defined on a neighbourhood $U_0$ of $(u_0,v_0) \in U_0$ if and only if, it satisfies the compatibility condition:

\begin{equation}
\frac{\partial \mathbf{S}}{\partial v} - \frac{\partial \mathbf{T}}{\partial u} + [\mathbf{S},\mathbf{T}]=0 \label{compatibility2},
\end{equation} 
where $[\mathbf{S},\mathbf{T}]=\mathbf{S}\mathbf{T}-\mathbf{T}\mathbf{S}$ is the Lie bracket.
\end{coro}
\begin{proof}
Identifying $\mathcal{M}_{n\times n}(\R)\equiv \R^{n^2}$ and defining $\mathbf{\Theta}(u,v,\mathbf{X}):=\mathbf{S}\mathbf{X}$ and $\mathbf{\Xi}(u,v,\mathbf{X}):=\mathbf{T}\mathbf{X}$ for $\mathbf{X}\in \mathcal{M}_{n\times n}(\R)$, the compatibility condition (\ref{compatibility1}) is equivalent to (\ref{compatibility2}) and by theorem \ref{frobg} follows the result.
\end{proof}
\section{Characterizing a frontal and its fundamental forms}

\begin{prop}\label{TMB}
	Let $\x:U \to \R^3$ be a smooth map with $U\subset \R^2$ an open set.Then, $\x$ is a frontal if and only if, for all $\p \in U$ there is a tangent moving base $\Omegam:V_p \to \mathcal{M}_{3\times2}(\R)$ of $\x$ with $V_p \subset U$ a neighborhood of $\p$. 
\end{prop}
\begin{proof}
	If $\x$ is a frontal, then for all $\p \in U$ there exists a unitary vector field $\num: V_p \to \R^3$ with $\x_u\cdot\n=0$, $\x_v\cdot\n=0$, $V_p$ a neighborhood of $\p$ which we can reduce in order to get $\nu_i\neq0$ on $V_p$ for any $i \in \{1,2,3\}$. Without loss of generality let us suppose that $\nu_1\neq0$ and define $\Omegam:=\begin{pmatrix}\w_1&\w_2\end{pmatrix}$ with $\w_1=(\nu_2,-\nu_1,0)$ and $\w_2=(\nu_3,0,-\nu_1)$. Since $\w_1$ and $\w_2$ are linearly independent, orthogonal to $\num$ and $dim(\num^\bot)=2$ ($\num^\bot$ orthogonal space to $\num$), we have that $\spann{\w_1}{\w_2}=\num^\bot$. Therefore, $\Omegam:V_p \to \mathcal{M}_{3\times2}(\R)$ is a tangent moving base of $\x$. The converse, just define $\num:=\frac{\w_1\times\w_2}{\|\w_1\times\w_2\|}$ taking $\w_1$ and $\w_2$ the columns from a tangent moving base $\Omegam:V_p \to \mathcal{M}_{3\times2}(\R)$. Then, $\num$ is orthogonal to $\x_u$ and $\x_v$ which belong to $\spann{\w_1}{\w_2}$.   
\end{proof}

\begin{prop}\label{OD}
	Let $\x:U \to \R^3$ be a smooth map with $U\subset \R^2$ an open set.Then, $\x$ is a frontal if and only if, for all $\p \in U$ there are smooth maps $\Omegam:V_p \to \mathcal{M}_{3\times2}(\R)$ and $\Lambdam:V_p \to \mathcal{M}_{2\times2}(\R)$ with $rank(\Omegam)=2$, $V_p \subset U$ a neighbourhood of $\p$, such that $D\x(\mathbf{q})=\Omegam\Lambdam^T$ for all $\mathbf{q} \in V_p$.   
\end{prop}
\begin{proof}
If $\x$ is a frontal, by proposition \ref{TMB} for all $\p \in U$ there is a tangent moving base $\Omegam:V_p \to \mathcal{M}_{3\times2}(\R)$ of $\x$ with $V_p \subset U$ a neighborhood of $\p$. Thus, there are coefficients $\lambda_{ij}$ such that $\x_u=\lambda_{11}\w_1+\lambda_{12}\w_2$ and $\x_v=\lambda_{21}\w_1+\lambda_{22}\w_2$. Therefore, $D\x(\mathbf{q})=\Omegam\Lambdam^T$ for all $\mathbf{q} \in V_p$ where $\Lambdam=(\lambda_{ij})$. Multiplying the equality by $\Omegam^T$ and as $\I_\Omega$ is invertible, we have that $\I^{-1}_\Omega\Omegam^T D\x(\mathbf{q})=\Lambdam^T$. Then, $\Lambdam:V_p \to \mathcal{M}_{2\times2}(\R)$ is smooth. Reciprocally, if we have $D\x(\mathbf{q})=\Omegam\Lambdam^T$ for all $\mathbf{q} \in V_p$, then $\x_u=\lambda_{11}\w_1+\lambda_{12}\w_2$ and $\x_v=\lambda_{21}\w_1+\lambda_{22}\w_2$. Hence $\x_u,\x_v \in \spann{\w_1}{\w_2}$ and as $Rank(\Omegam)=2$, $\Omegam$ is a tangent moving base of $\x$. By proposition \ref{TMB} $\x$ is a frontal.       
\end{proof}

\begin{rem}
In the proof of proposition (\ref{OD}), observe that $\Lambdam=D\x^T\Omegam(\I_\Omega^T)^{-1}$, then $\Lambdam$ is determined by a local tangent moving base of $\x$. Also having a decomposition $D\x=\Omegam\Lambdam^T$ with $rank(\Omegam)=2$ implies that $\Omegam$ is a tangent moving base of $\x$. 
\end{rem}
From now on, as we want to describe local properties and tangent moving bases exist locally, we can suppose that we have a global tangent moving base for a frontal restringing the domain if necessary. If $\x$ is a frontal and $\Omegam$ a tangent moving base of $\x$, we denote $\Lambdam:=D\x^T\Omegam(\I_\Omega)^{-1}$, $\laO:=det(\Lambdam)$ and $\mathfrak{T}_\Omega$ as the principal ideal generated by $\laO$ in the ring $C^\infty(U,\R)$. Thus, we have globally $D\x=\Omegam\Lambdam^T$, $\Sigma(\x)=\laO^{-1}(0)$ and $rank(D\x)=rank(\Lambdam)$. Also, with a tangent moving base $\Omegam=\begin{pmatrix}\w_1&\w_2\end{pmatrix}$ given, we always choose as unit normal vector field $\num:U \to \R^3$ along $\x$, the induced by $\Omegam$ (i.e $\n=\frac{\w_1\times\w_2}{\|\w_1\times\w_2\|}$). 

\begin{definition}
	Let $\x:U \to \R^3$ be a frontal, $\Omegam=\begin{pmatrix}\w_1&\w_2\end{pmatrix}$ and $\bar{\Omegam}=\begin{pmatrix}\bar{\w}_1&\bar{\w}_2\end{pmatrix}$ tangent moving bases of $\x$. We say that $\Omegam$ and $\bar{\Omegam}$ are {\it compatibles} if $\w_{1}\times\w_{2}\cdot\bar{\w}_{1}\times\bar{\w}_{2}>0$. Also, $\Omegam$ is {\it orthonormal} tangent moving base if $|\w_{1}|=|\w_{2}|=1$ and $\w_1\cdot \w_2=0$.  	
\end{definition}
 
\begin{teo}\label{D}
	Let $\x:U \to \R^3$ be a frontal and $\Omegam$ a tangent moving base of $\x$, then the matrices defined by equations \ref{FF1} and \ref{FF2} have the following decomposition:
	\begin{subequations}
		\label{FFSp}
		\begin{align}
		\begin{pmatrix}
		E&F\\
		F&G
		\end{pmatrix}&=\begin{pmatrix}
		\la{1}{1} & \la{1}{2}\\
		\la{2}{1} & \la{2}{2}
		\end{pmatrix}\begin{pmatrix}
		\EO&\FO\\
		\FO&\GO
		\end{pmatrix}\begin{pmatrix}
		\la{1}{1} & \la{1}{2}\\
		\la{2}{1} & \la{2}{2}
		\end{pmatrix}^T\label{FFS1p}
		\\ 
		\begin{pmatrix}
		e&f\\
		f&g
		\end{pmatrix}&=\begin{pmatrix}
		\la{1}{1} & \la{1}{2}\\
		\la{2}{1} & \la{2}{2}
		\end{pmatrix}\begin{pmatrix}
		\eo&\fuo\\
		\fdo&\go
		\end{pmatrix}\label{FFS2p},
		\end{align}
	\end{subequations}	
	in which all the components are smooth real functions defined on $U$, $\EO>0$, $\GO>0$, $\EO\GO-\FO^2>0$, $rank(D\x)=rank(\Lambdam)$, $\Sigma(\x)=\laO^{-1}(0)$ and
	\begin{subequations}
	\begin{align}
	\Lambdam_{(1)u}^{\phantomsection}\I_\Omega\Lambdam_{(2)}^T-\Lambdam_{(1)}^{\phantomsection}\I_\Omega\Lambdam_{(2)u}^T+E_v-F_u \in \mathfrak{T}_\Omega\label{li1}\\
	\Lambdam_{(1)v}^{\phantomsection}\I_\Omega\Lambdam_{(2)}^T-\Lambdam_{(1)}^{\phantomsection}\I_\Omega\Lambdam_{(2)v}^T+F_v-G_u \in \mathfrak{T}_\Omega\label{li2}
	\end{align}
	\end{subequations}
	where $\Lambdam=(\lambda_{ij})$. 
\end{teo}
\begin{proof}
	We have $D\x=\Omegam\Lambdam^T$, then $\I=D\x^T D\x=\Lambdam\Omegam^T\Omegam\Lambdam^T=\Lambdam\I_\Omega\Lambdam^T$. Also, $\II=-D\x^T D\n=\Lambdam(-\Omegam^TD\n)=\Lambdam\II_\Omega$. Now, let us set the skew-symmetric matrices:
	$$
	\mathbf{A}_1:=\begin{pmatrix}
	0&-(E_v-F_u)\\
	E_v-F_u&0
	\end{pmatrix}, \ \mathbf{B}_1:=\begin{pmatrix}
	0&-\tau_1\\
	\tau_1&0
	\end{pmatrix}:=\Omegam_u^T\Omegam-\Omegam^T\Omegam_u.$$
	From (\ref{GA1}) and (\ref{T1X}) we have $D\x_u^TD\x-\frac{1}{2}\I_u=\frac{1}{2}\mathbf{A}_1$, then using that $\I=\Lambdam\I_\Omega\Lambdam^T$, $D\x=\Omegam\Lambdam^T$ and developing derivatives, 
	$$(\Lambdam\Omegam_u^T+\Lambdam_u\Omegam^T)\Omegam\Lambdam^T=\frac{1}{2}(\Lambdam_u\I_\Omega\Lambdam^T+\Lambdam\I_{\Omega u}\Lambdam^T+\Lambdam \I_\Omega\Lambdam_u^T)+\frac{1}{2}\mathbf{A}_1.$$
	Substituting $\I_\Omega=\Omegam^T\Omegam$ and $\I_{\Omega u}=\Omegam_u^T\Omegam+\Omegam^T\Omegam_u$, we can group and cancel similar terms, getting
	$$\Lambdam\mathbf{B}_1\Lambdam^T=\Lambdam\I_\Omega\Lambdam_u^T-\Lambdam_u\I_\Omega\Lambdam^T+\mathbf{A}_1.$$
	multiplying the equality by left side with $\begin{pmatrix}
	1&0\end{pmatrix}$ and by the right side with $\begin{pmatrix}
	0&1\end{pmatrix}^T$, we obtain,
	$$-\tau_1\laO=\Lambdam_{(1)}^{\phantomsection}\begin{pmatrix}
	0&-\tau_1\\
	\tau_1&0
	\end{pmatrix}\Lambdam_{(2)}^T=\Lambdam_{(1)}^{\phantomsection}\I_\Omega\Lambdam_{(2)u}^T-\Lambdam_{(1)u}^{\phantomsection}\I_\Omega\Lambdam_{(2)}^T-(E_v-F_u)$$ and from it follows (\ref{li1}). Setting the matrices:
	$$
	\mathbf{A}_2:=\begin{pmatrix}
	0&-(F_v-G_u)\\
	F_v-G_u&0
	\end{pmatrix}, \ \mathbf{B}_2=\begin{pmatrix}
	0&-\tau_2\\
	\tau_2&0
	\end{pmatrix}:=\Omegam_v^T\Omegam-\Omegam^T\Omegam_v$$ Observing that, $D\x_v^TD\x-\frac{1}{2}\I_v=\frac{1}{2}\mathbf{A}_2$ and proceeding similarly as before, it follows (\ref{li2}). 
	
\end{proof}
The conditions (\ref{li1}) and (\ref{li2}) in theorem \ref{D} may seem kind of strange, but we will see in proposition \ref{metric} why these are so important. Also these expressions can be reduced depending on the type of $\Omegam$ chosen. If we have a tangent moving base of a frontal, we always can construct an orthonormal one applying Gram-Schmidt orthonormalization, then the decompositions in theorem \ref{D} are reduced and follows easily the corollary:

\begin{coro}\label{Dc}
	Let $\x:U \to \R^3$ be a frontal and $\Omegam$ a orthonormal tangent moving base of $\x$, then the matrices defined by equations \ref{FF1} and \ref{FF2} have the following decomposition:
	\begin{subequations}
		\label{FFSpc}
		\begin{align}
		\begin{pmatrix}
		E&F\\
		F&G
		\end{pmatrix}&=\begin{pmatrix}
		\la{1}{1} & \la{1}{2}\\
		\la{2}{1} & \la{2}{2}
		\end{pmatrix}\begin{pmatrix}
		\la{1}{1} & \la{1}{2}\\
		\la{2}{1} & \la{2}{2}
		\end{pmatrix}^T\label{FFS1pc}
		\\ 
		\begin{pmatrix}
		e&f\\
		f&g
		\end{pmatrix}&=\begin{pmatrix}
		\la{1}{1} & \la{1}{2}\\
		\la{2}{1} & \la{2}{2}
		\end{pmatrix}\begin{pmatrix}
		\eo&\fuo\\
		\fdo&\go
		\end{pmatrix}\label{FFS2pc},
		\end{align}
	\end{subequations}	
	in which all the components are smooth real functions defined on $U$, $rank(D\x)=rank(\Lambdam)$, $\Sigma(\x)=\laO^{-1}(0)$ and
	\begin{subequations}
		\begin{align}
		(\Lambdam_{(1)}^{\phantomsection}\Lambdam_{(1)}^T)_v-2\Lambdam_{(1)}^{\phantomsection}\Lambdam_{(2)u}^T \in \mathfrak{T}_\Omega\label{li1c}\\
		2\Lambdam_{(1)v}^{\phantomsection}\Lambdam_{(2)}^T-(\Lambdam_{(2)}^{\phantomsection}\Lambdam_{(2)}^T)_u \in \mathfrak{T}_\Omega\label{li2c}
		\end{align}
	\end{subequations}
	where $\Lambdam=(\lambda_{ij})$. 
\end{coro}

\begin{rem}\label{specialbase}
If $\x:U \to \R^3$ is a frontal and $\Omegam$ a tangent moving base of $\x$, we can find a tangent moving base $\hat{\Omegam}$ having one of the following forms:
$$
\begin{pmatrix}
1&0\\
0&1\\
g_1&g_2
\end{pmatrix}, \ \begin{pmatrix}
1&0\\
g_1&g_2\\
0&1
\end{pmatrix}, \ \begin{pmatrix}
0&1\\
1&0\\
g_1&g_2
\end{pmatrix},
$$
$$
\begin{pmatrix}
g_1&g_2\\
1&0\\
0&1
\end{pmatrix}, \ \begin{pmatrix}
g_1&g_2\\
0&1\\
1&0
\end{pmatrix}, \ \begin{pmatrix}
0&1\\
g_1&g_2\\
1&0
\end{pmatrix},
$$
\\
with $g_1,g_2: U \to \R$ smooth functions and the matrix $\hat{\Lambdam}^T$ as an {\it exact differential}, it means, there is a smooth map $(a,b):U \to \R^2$ such that $D(a,b)=\hat{\Lambdam}^T$. To see it, as the columns of $\Omegam$ are linearly independent, then applying reduction of Gauss-Jordan with a finite number of operations by columns, it can be reduced to one of the forms above. Without loss of generality, let us suppose it is reduced to the first one. If we denote $\mathbf{E}_1,\mathbf{E}_2,..,\mathbf{E}_m$ the elementary matrices $2\times2$ corresponding to the operations by columns, we have:
$$D\x=\Omegam\Lambdam^T=\Omegam \mathbf{E}_1\mathbf{E}_2\cdots \mathbf{E}_m\mathbf{E}_m^{-1}\cdots \mathbf{E}_2^{-1}\mathbf{E}_1^{-1}\Lambdam^T=\begin{pmatrix}
1&0\\
0&1\\
g_1&g_2
\end{pmatrix}\mathbf{E}_m^{-1}\cdots \mathbf{E}_2^{-1}\mathbf{E}_1^{-1}\Lambdam^T$$
Denoting $\hat{\Lambdam}^T:=\mathbf{E}_m^{-1}\cdots \mathbf{E}_2^{-1}\mathbf{E}_1^{-1}\Lambdam^T$ and $\x=(a,b,c)$, we can multiply the last equality by $\begin{pmatrix}
	1&0&0\\
	0&1&0\\
\end{pmatrix}$ to get:
$$D(a,b)=\begin{pmatrix}
a_u & a_v\\
b_u & b_v
\end{pmatrix}=\begin{pmatrix}
1&0&0\\
0&1&0\\
\end{pmatrix}D\x=\begin{pmatrix}
1&0&0\\
0&1&0\\
\end{pmatrix}\begin{pmatrix}
1&0\\
0&1\\
g_1&g_2
\end{pmatrix}\hat{\Lambdam}^T=\id_2\hat{\Lambdam}^T=\hat{\Lambdam}^T.$$ On the other hand, a simple computation leads to
\begin{align}
\I_{\hat{\Omega}}=\begin{pmatrix}
1+g_1^2 & g_1g_2\\
g_1g_2  & 1+g_2^2
\end{pmatrix},\II_{\hat{\Omega}}=\begin{pmatrix}
g_{1u} & g_{1v}\\
g_{2u} & g_{2v}		\end{pmatrix}(1+g_1^2+g_2^2)^{-\frac{1}{2}},\nonumber
\end{align}
and since $D\n=\hat{\Omegam}\mum^T$ with $\n=(-g_1,-g_2,1)det(\I_\Omega)^{-\frac{1}{2}}$, reasoning as before we get that $D(-g_1det(\I_\Omega)^{-\frac{1}{2}},-g_2det(\I_\Omega)^{-\frac{1}{2}})=\mum^T$.
\end{rem}
By this fact and theorem \ref{D}, follows the result:  
\begin{coro}\label{D2c}
	Let $\x:U \to \R^3$ be a frontal and $\Omegam$ a tangent moving base of $\x$ with the form of remark \ref{specialbase}, then the matrices defined by equations \ref{FF1} and \ref{FF2} have a decomposition in this form:
	\begin{subequations}
		\label{DE}
		\begin{align}
		\begin{pmatrix}
		E&F\\
		F&G
		\end{pmatrix}&=\begin{pmatrix}
		a_u & b_u\\
		a_v & b_v
		\end{pmatrix}\begin{pmatrix}
		1+g_1^2 & g_1g_2\\
		g_1g_2  & 1+g_2^2
		\end{pmatrix}\begin{pmatrix}
		a_u & b_u\\
		a_v & b_v
		\end{pmatrix}^T\label{DE1}
		\\ 
		\begin{pmatrix}
		e&f\\
		f&g
		\end{pmatrix}&=\begin{pmatrix}
		a_u & b_u\\
		a_v & b_v
		\end{pmatrix}\begin{pmatrix}
		g_{1u} & g_{1v}\\
		g_{2u} & g_{2v}		\end{pmatrix}(1+g_1^2+g_2^2)^{-\frac{1}{2}}
		\label{DE2}
		\end{align}
	\end{subequations}
	in which $g_1$, $g_2$, $a$ and $b$ are smooth real functions defined in $U$. In particular, (\ref{DE2}) implies $(a,b)_u\cdot(g_1,g_2)_v=(a,b)_v\cdot(g_1,g_2)_u$. 
	
\end{coro}

\begin{ex}
	The cuspidal cross-cap (see Figure \ref{ccrosscap}) can be decomposed in this way:
	\begin{figure}[h]
		\begin{center}
			\includegraphics[scale=0.22]{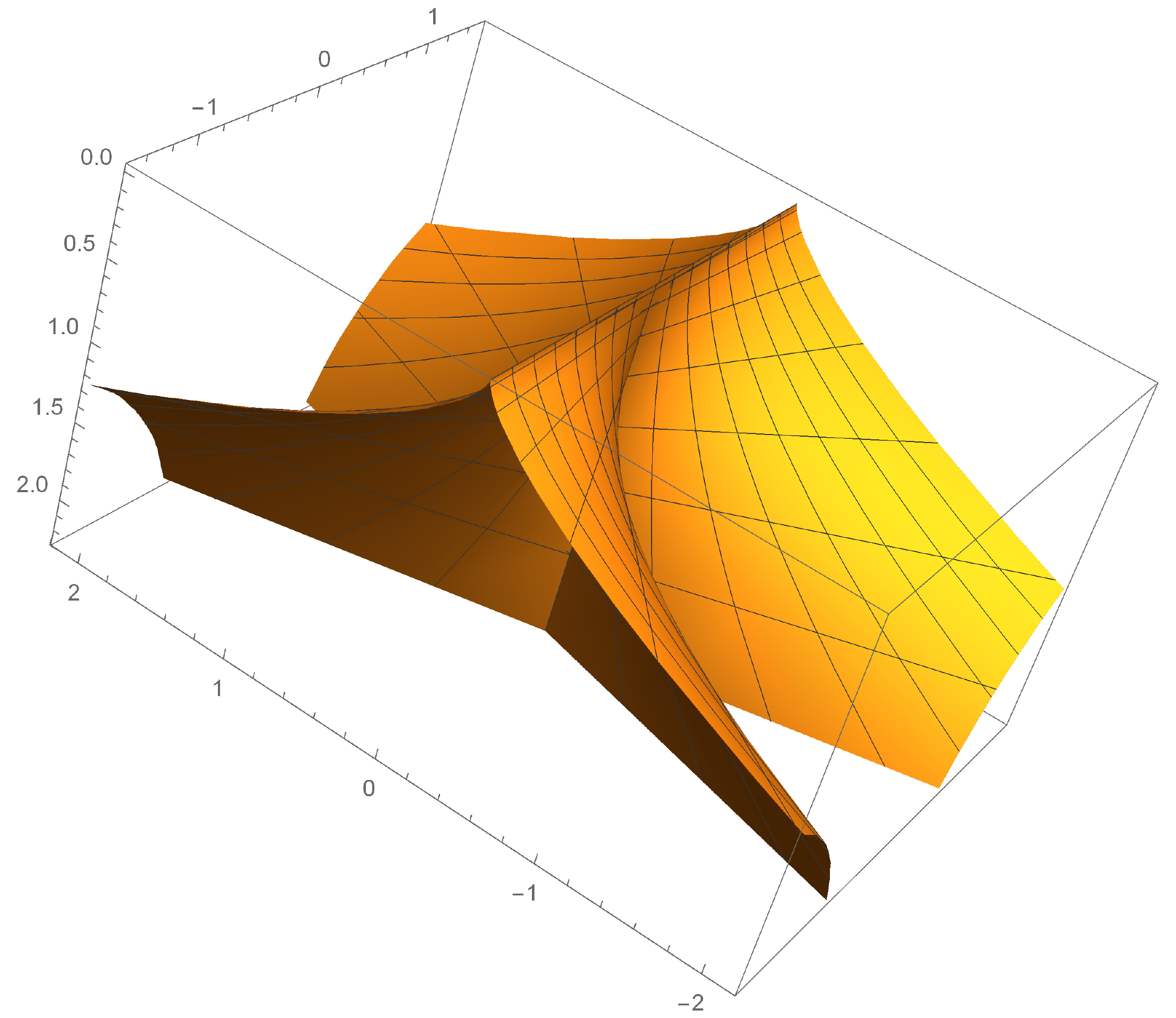}\qquad\qquad \includegraphics[scale=0.22]{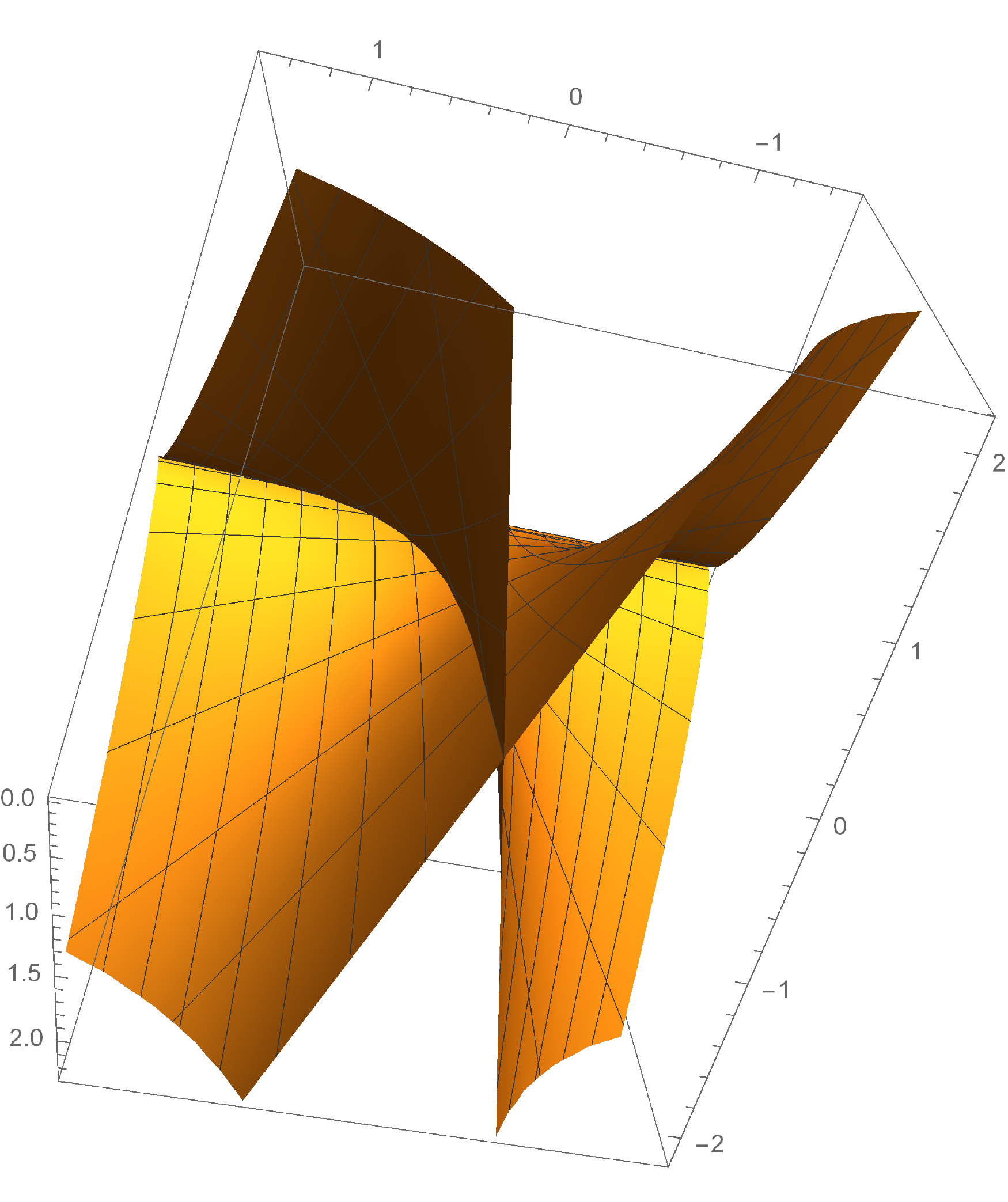}
		\end{center}
		\caption{The cuspidal cross-cap $(\x(u,v)=(u,v^2, uv^3))$, an example of a proper frontal which is not a front \cite{FSUY}.}\label{ccrosscap}
	\end{figure}
	\begin{subequations}
		\label{CCCD}
		\begin{align}
		&D\x=\begin{pmatrix}
		1&0\\
		0&1\\
		v^3&\frac{3}{2}uv
		\end{pmatrix}\begin{pmatrix}
		1&0\\
		0&2v
		\end{pmatrix}=\Omegam\Lambdam^T,\text{ where }\Omegam=\begin{pmatrix}
		1&0\\
		0&1\\
		v^3&\frac{3}{2}uv
		\end{pmatrix}, \Lambdam=\begin{pmatrix}
		1&0\\
		0&2v
		\end{pmatrix}\nonumber\\
		&\begin{pmatrix}
		E&F\\
		F&G
		\end{pmatrix}=\begin{pmatrix}
		1 & 0\\
		0 & 2v
		\end{pmatrix}\begin{pmatrix}
		1+v^6 & \frac{3}{2}uv^4\\
		\frac{3}{2}uv^4 & 1+\frac{9}{4}u^2v^2
		\end{pmatrix}\begin{pmatrix}
		1 & 0\\
		0 & 2v
		\end{pmatrix}^T\nonumber
		\\ 
		&\begin{pmatrix}
		e&f\\
		f&g
		\end{pmatrix}=\begin{pmatrix}
		1 & 0\\
		0 & 2v
		\end{pmatrix}\begin{pmatrix}
		0 & 3v^2\\
		\frac{3}{2}v & \frac{3}{2}u
		\end{pmatrix}\frac{1}{\sqrt{1+v^6+\frac{9}{4}u^2v^2}}\nonumber
		\end{align}
	\end{subequations}
\end{ex}	

\begin{teo}\label{D2}
	Let $\I:U \to \mathcal{M}_{2\times2}(\R)$ be a smooth map, with $\I$ decomposing in this form:

		$$\I=\begin{pmatrix}
		a_u & b_u\\
		a_v & b_v
		\end{pmatrix}\begin{pmatrix}
		1+g_1^2 & g_1g_2\\
		g_1g_2  & 1+g_2^2
		\end{pmatrix}\begin{pmatrix}
		a_u & b_u\\
		a_v & b_v
		\end{pmatrix}^T
		$$
in which $g_1$, $g_2$, $a$ and $b$  are smooth real functions defined in $U$, satisfying $(a,b)_u\cdot(g_1,g_2)_v=(a,b)_v\cdot(g_1,g_2)_u$. Then, for each $(u_0,v_0)\in U$ and $\p \in \R^3$, there is a frontal $\x:V \to \R^3$, $V \subset U$, $V$ a neighborhood of $(u_0,v_0)$ with first fundamental form $\I$ and second fundamental form $D(a,b)^T D(g_1,g_2)(1+g_1^2+g_2^2)^{-\frac{1}{2}}$.
	
\end{teo}
\begin{proof}
	Setting the matrices:\\
	$$\Omegam:=\begin{pmatrix}
	1 & 0\\
	0 & 1\\
	g_1 & g_2
	\end{pmatrix}, \Lambdam^T:=\begin{pmatrix}
	a_u & a_v\\
	b_u & b_v
	\end{pmatrix}, 
	\mathbf{e}_{1}:=\begin{pmatrix}
	1\\
	0
	\end{pmatrix}, \mathbf{e}_{2}:=\begin{pmatrix}
	0\\
	1
	\end{pmatrix}$$ 
	as $(a,b)_u\cdot(g_1,g_2)_v=(a,b)_v\cdot(g_1,g_2)_u$, then $$\Omegam_u\Lambdam^T\ed=\begin{pmatrix}
0&0\\
0&0\\
g_{1u}&g_{2u}
\end{pmatrix}\begin{pmatrix}
a_v\\
b_v
\end{pmatrix}=\begin{pmatrix}
	0&0\\
	0&0\\
	g_{1v}&g_{2v}
\end{pmatrix}\begin{pmatrix}
	a_u\\
	b_u
\end{pmatrix}=\Omegam_v\Lambdam^T\eu$$
on the other hand, since $\Lambdam^T$ is an exact differential, $\Lambdam_u^T\ed=\Lambdam_v^T\eu$. Thus, $\Omegam\Lambdam_u^T\ed=\Omegam\Lambdam_v^T\eu$ and adding this equality to the above one, we get:
$$(\Omegam\Lambdam^T)_u\ed=\Omegam_u\Lambdam^T\ed+\Omegam\Lambdam_u^T\ed=\Omegam_v\Lambdam^T\eu+\Omegam\Lambdam_v^T\eu=(\Omegam\Lambdam^T)_v\eu$$
Denoting by $\mathbf{z}_1$ and $\mathbf{z}_2$ the first and second columns of $\Omegam\Lambdam^T$ respectively, fixing $(u_0,v_0)\in U$ and $\p \in \R^3$ the last equality is equivalent to $\mathbf{z}_{2u}=\mathbf{z}_{1v}$, which is the compatibility condition of the system:
\begin{subequations}
	\begin{align}
	\x_u&=\mathbf{z}_1
	\\
	\x_v&=\mathbf{z}_2
	\\
	\x(u_0,v_0)&=\p,
	\end{align}	
\end{subequations}
By theorem \ref{frobg}, this system of PDE has a solution $\x:V \to \R^3$, $V \subset U$, $V$ a neighborhood of $(u_0,v_0)$. Therefore $D\x=\Omegam\Lambdam^T$ and by proposition \ref{OD}, $\x:V \to \R^3$ is a frontal. Now, the first fundamental form is $D\x^TD\x=\Lambdam\Omegam^T\Omegam\Lambdam^T=\I$ as we wished. Using that $\n=(-g_1,-g_2,1)(1+g_1^2+g_2^2)^{-\frac{1}{2}}$ and (\ref{SFO}), the second fundamental form is $\Lambdam\II_\Omega=D(a,b)^T D(g_1,g_2)(1+g_1^2+g_2^2)^{-\frac{1}{2}}$. 
\end{proof}
\begin{prop}\label{propE}
Let $\x:U \to \R^3$ be a frontal and $\Omegam$ a tangent moving base of $\x$, then the matrices $\TA{1}, \TA{2}$ satisfies 
$\I_\Omega\TA{1}^T+\TA{1}\I_\Omega=\I_{\Omega u}$ and $\I_\Omega\TA{2}^T+\TA{2}\I_\Omega=\I_{\Omega v}$.	
\end{prop}
\begin{proof}
$\I_{\Omega u}=\Omegam_u^T\Omegam+\Omegam^T\Omegam_u=\Omegam_u^T\Omegam\I_\Omega^{-1}\I_\Omega+\I_\Omega\I_\Omega^{-1}\Omegam^T\Omegam_u=\TA{1}\I_\Omega+\I_\Omega\TA{1}^T$. For $\I_{\Omega v}$ is analogous.	
\end{proof}	
\begin{prop}
	Let $\x:U \to \R^3$ be a proper frontal and $\Omegam$ a tangent moving base of $\x$, then the Christoffel symbols defined on $U-\laO^{-1}(0)$ have the following decomposition:
	\begin{align}
	\GA{1}=(\Lambdam\TA{1}+\Lambdam_u)\Lambdam^{-1} \text{ and } \GA{2}=(\Lambdam\TA{2}+\Lambdam_v)\Lambdam^{-1}\nonumber
	\end{align}
\end{prop}
\begin{proof}
	$\GA{1}=(D\x_u^{T}D\x)\I^{-1}=((\Omegam_u\Lambdam^T+\Omegam\Lambdam_u^T)^T\Omegam\Lambdam^T)(\Lambdam^{T})^{-1}\I_\Omega^{-1}\Lambdam^{-1}\\
	=(\Lambdam\Omegam_u^T+\Lambdam_u\Omegam^T)\Omegam\Lambdam^T(\Lambdam^{T})^{-1}\I_\Omega^{-1}\Lambdam^{-1}=(\Lambdam\Omegam_u^T\Omegam\I_\Omega^{-1}+\Lambdam_u\Omegam^T\Omegam\I_\Omega^{-1})\Lambdam^{-1}$. Since $\TA{1}=\Omegam_u^T\Omegam\I_\Omega^{-1}$ and $\I_\Omega=\Omegam^T\Omegam$ we have the result. For $\GA{2}$ it is analogous.
\end{proof}
\begin{rem}\label{rem1}
	With this decomposition of the Christoffel symbols, by density of non-singular points and smoothness of $\TA{i}$ on $U$, we get that $\TA{1}$ and $\TA{2}$ can be expressed by:
	\begin{itemize}
		\item For $\p\in \Sigma(\x)^c$,
		\begin{align}
		\TA{1}=\Lambdam^{-1}(\GA{1}\Lambdam-\Lambdam_u)\text{ and }\TA{2}=\Lambdam^{-1}(\GA{2}\Lambdam-\Lambdam_v)\nonumber.
		\end{align}
		\item For $\p\in \Sigma(\x)$,
		\begin{align}
		\TA{1}=\lim\limits_{(u,v)\to p}\Lambdam^{-1}(\GA{1}\Lambdam-\Lambdam_u)\text{ and }\TA{2}=\lim\limits_{(u,v)\to p}\Lambdam^{-1}(\GA{2}\Lambdam-\Lambdam_v)\nonumber.
		\end{align}

	\end{itemize}
	Where the right sides are restricted to the open set $\Sigma(\x)^c$. As $\GA{1}$ and $\GA{2}$ are expressed in terms of $E$, $F$, $G$ and these by (\ref{FFS1p}) are expressed in terms of $\EO$, $\FO$, $\GO$ and $\lambda_{ij}$, then $\TA{1}$ and $\TA{2}$ can be expressed just using $\EO$, $\FO$, $\GO$ and $\lambda_{ij}$ on $\Sigma(\x)^c$. By density, these are completely determined by $\EO$, $\FO$, $\GO$ and $\lambda_{ij}$ on $U$.   
\end{rem}
\begin{prop}\label{metric}
	Let $\I, \I_\Omega, \Lambdam:U \to \mathcal{M}_{2\times2}(\R)$ arbitrary smooth maps, $\I_\Omega$ symmetric non-singular, $\laO=det(\Lambdam)$ and $\mathfrak{T}_\Omega$ the principal ideal generated by $\laO$ in the ring $C^\infty(U,\R)$. If we have,
	\begin{align}
	\I=\begin{pmatrix}
	E&F\\
	F&G
	\end{pmatrix}=\Lambdam\I_\Omega\Lambdam^T\label{equi}
	\end{align}
	with $int(\laO^{-1}(0))=\emptyset$ and if we define $\GA{1}$ by (\ref{GA1}) and $\GA{2}$ by (\ref{GA2}) on $U-\laO^{-1}(0)$, then the maps
	\begin{subequations}
		\begin{align}
		&\Lambdam^{-1}(\GA{1}\Lambdam-\Lambdam_u),\label{espr1}\\  &\Lambdam^{-1}(\GA{2}\Lambdam-\Lambdam_v),\label{espr2}
		\end{align}
	\end{subequations}
	defined on $U-\laO^{-1}(0)$, have unique $C^{\infty}$ extensions to $U$ if and only if,
	\begin{subequations}
		\begin{align}
		\Lambdam_{(1)u}^{\phantomsection}\I_\Omega\Lambdam_{(2)}^T-\Lambdam_{(1)}^{\phantomsection}\I_\Omega\Lambdam_{(2)u}^T+E_v-F_u \in \mathfrak{T}_\Omega\label{cli1}\\
		\Lambdam_{(1)v}^{\phantomsection}\I_\Omega\Lambdam_{(2)}^T-\Lambdam_{(1)}^{\phantomsection}\I_\Omega\Lambdam_{(2)v}^T+F_v-G_u \in \mathfrak{T}_\Omega\label{cli2}
		\end{align}
	\end{subequations} 
\end{prop}
\begin{proof}
	For the necessary condition, let us set the skew-symmetric matrix
	$$
	\mathbf{A}_1:=\begin{pmatrix}
	0&-(E_v-F_u)\\
	E_v-F_u&0
	\end{pmatrix}$$
	and suppose that $\TA{1}$ is the $C^{\infty}$ extension of $\Lambdam^{-1}(\GA{1}\Lambdam-\Lambdam_u)$, then
	\begin{align}
	\Lambdam\TA{1}=\GA{1}\Lambdam-\Lambdam_u
	\end{align}
	on $U-\laO^{-1}(0)$, hence using (\ref{GA1}) we have
	\begin{align}
	\Lambdam\TA{1}=(\frac{1}{2}\I_u+\frac{1}{2}\mathbf{A}_1)\I^{-1}\Lambdam-\Lambdam_u.
	\end{align}
	Substituting $\I$ and $\I_u$ in the last equality using (\ref{equi}) and multiplying by the right side with $2\I_\Omega\Lambdam^T$, operating some terms we can get,
	\begin{align}
	\Lambdam(2\TA{1}\I_\Omega-\I_{\Omega u})\Lambdam^T=\Lambdam\I_\Omega\Lambdam_u^T-\Lambdam_u\I_\Omega\Lambdam^T+\mathbf{A}_1.\label{equi2}
	\end{align}
	Observe that, the right side of (\ref{equi2}) is skew-symmetric, then $2\TA{1}\I_\Omega-\I_{\Omega u}$ as well and since $U-\laO^{-1}(0)$ is dense, this is also true on $U$. Thus, 
	$$2\TA{1}\I_\Omega-\I_{\Omega u}=\begin{pmatrix}
		0&-\tau_1\\
		\tau_1&0
	\end{pmatrix}$$
	for any $\tau_1\in C^\infty(U,\R)$ and since the equality (\ref{equi2}) is valid on $U$ by density, then multiplying this by left side with $\begin{pmatrix}
	1&0\end{pmatrix}$ and by the right side with $\begin{pmatrix}
	0&1\end{pmatrix}^T$, we obtain,
	$$-\tau_1\laO=\Lambdam_{(1)}^{\phantomsection}\begin{pmatrix}
	0&-\tau_1\\
	\tau_1&0
	\end{pmatrix}\Lambdam_{(2)}^T=\Lambdam_{(1)}^{\phantomsection}\I_\Omega\Lambdam_{(2)u}^T-\Lambdam_{(1)u}^{\phantomsection}\I_\Omega\Lambdam_{(2)}^T-(E_v-F_u)$$ and from it follows (\ref{cli1}). Setting the matrix:
	$$
	\mathbf{A}_2:=\begin{pmatrix}
	0&-(F_v-G_u)\\
	F_v-G_u&0
	\end{pmatrix}$$ and observing that $\GA{2}=(\frac{1}{2}\I_v+\frac{1}{2}\mathbf{A}_2)\I^{-1}$, proceeding similarly as before, it follows (\ref{cli2}). For the sufficient condition, if we have (\ref{cli1}), (\ref{cli2}), as $U-\laO^{-1}(0)$ is dense then there exist unique $\tau_1, \tau_2 \in C^\infty(U,\R)$ such that, 
	$$\Lambdam_{(1)u}^{\phantomsection}\I_\Omega\Lambdam_{(2)}^T-\Lambdam_{(1)}^{\phantomsection}\I_\Omega\Lambdam_{(2)u}^T+E_v-F_u=\laO\tau_1,$$
	$$\Lambdam_{(1)v}^{\phantomsection}\I_\Omega\Lambdam_{(2)}^T-\Lambdam_{(1)}^{\phantomsection}\I_\Omega\Lambdam_{(2)v}^T+F_v-G_u=\laO\tau_2.$$
	Defining the smooth maps on $U$,
	\begin{align}\label{deftau}
	\TA{1}:=\frac{1}{2}(\begin{pmatrix}
	0&-\tau_1\\
	\tau_1&0
	\end{pmatrix}+\I_{\Omega u})\I_\Omega^{-1}\text{ and }
	\TA{2}:=\frac{1}{2}(\begin{pmatrix}
	0&-\tau_2\\
	\tau_2&0
	\end{pmatrix}+\I_{\Omega v})\I_\Omega^{-1},
	\end{align}
	we have that 
	$$\Lambdam(2\TA{1}\I_\Omega-\I_{\Omega u})\Lambdam^T=\Lambdam\I_\Omega\Lambdam_u^T-\Lambdam_u\I_\Omega\Lambdam^T+\mathbf{A}_1,$$
	$$\Lambdam(2\TA{2}\I_\Omega-\I_{\Omega v})\Lambdam^T=\Lambdam\I_\Omega\Lambdam_v^T-\Lambdam_v\I_\Omega\Lambdam^T+\mathbf{A}_2,$$ 
	which leads to $\TA{1}$ and $\TA{2}$ be equal to (\ref{espr1}) and (\ref{espr2}) respectively on $U-\laO^{-1}(0)$. Thus, by density and smoothness of $\TA{1}$ and $\TA{2}$, these are unique $C^{\infty}$-extensions.   	  
\end{proof}
\begin{rem}\label{tau}
By proposition \ref{metric}, we always can define the matrices $\TA{1}, \TA{2}$ by (\ref{deftau}) from a smooth map $\I:U \to \mathcal{M}_{2\times2}(\R)$ satisfying a decomposition as in (\ref{equi}) with the conditions (\ref{cli1}) and (\ref{cli2}). These maps $\TA{1},\TA{2}$ automatically satisfy the relationships of proposition \ref{propE} as also are the unique $C^\infty$ extension of (\ref{espr1}) and (\ref{espr2}). It is natural the question if a decomposition as in (\ref{equi}) implies the conditions (\ref{cli1}), (\ref{cli2}) and the answer is not. For example the matrix $\I$ associated to the first fundamental form of $(u,v^2,uv)$ (The Whitney cross-cap) is singular at $(0,0)$ and have a rank $\geq1$ on the entire $\R^2$, then you can obtain the Cholesky decomposition $\I=\Lambdam\Lambdam^T$ (here $\I_{\Omega}$ can be chosen as $\id_2$), where $\Lambdam:\R^2 \to \mathcal{M}_{2\times2}(\R)$ is smooth and a lower triangular matrix. It is not difficult to check that in this case the condition (\ref{cli1}) and (\ref{cli1}) are not satisfied for all neighborhood of $(0,0)$.  	
\end{rem}
\begin{definition}
	Let $\x:U \to \R^3$ be a frontal and $\Omegam$ a tangent moving base  of $\x$, we define the $\Omega$-{\it relative curvature} $\KO:=det(\boldsymbol{\mu})$ and the $\Omega$-{\it relative mean curvature} $H_\Omega:=-\frac{1}{2}tr(\mum adj(\Lambdam))$, where $tr()$ is the trace and $adj()$ is the adjoint of a matrix.
\end{definition} 
We are going to use $\KO$ and $\HO$ to characterize wave fronts in propositions \ref{wft} and \ref{wf}, but first we need to prove some propositions. The reason why we call these functions curvatures is in the following result. 

\begin{prop}\label{lim} 
	Let $\x:U \to \R^3$ be a proper frontal, $\Omegam$ a tangent moving base  of $\x$, $\KO$, $H_\Omega$, $K$ and $H$ the $\Omega$-relative curvature, the $\Omega$-relative mean curvature, the Gaussian curvature and the mean curvature of $\x$ respectively.Then,	
	\begin{itemize}
		\item for $\p\in \Sigma(\x)^c$, $\KO=\laO K$ and $H_\Omega=\laO H$,
		\item for $\p\in \Sigma(\x)$, $\KO=\lim\limits_{(u,v)\to p}\laO K$ and $H_\Omega=\lim\limits_{(u,v)\to p}\laO H$, 
		
	\end{itemize} 
where the right sides are restricted to the open set $\Sigma(\x)^c$.	
\end{prop}
\begin{proof}
	By theorem \ref{D}, $\I=\Lambdam\I_\Omega \Lambdam^T$ and $\II=\Lambdam\II_\Omega$, then for $\p\in \Sigma(\x)^c$, $\boldsymbol{\alpha}=-\II^T\I^{-1}=-\II_\Omega^T\Lambdam^T(\Lambdam^T)^{-1}\I_\Omega^{-1}\Lambdam^{-1}=\boldsymbol{\mu}\Lambdam^{-1}$. Thus, $\boldsymbol{\alpha}\Lambdam=\boldsymbol{\mu}$ and $\KO=det(\boldsymbol{\mu})=det(\boldsymbol{\alpha})det(\Lambdam)=\laO K$. Also, we have $\boldsymbol{\alpha}\laO=\boldsymbol{\mu}adj(\Lambdam)$, then $H_\Omega=-\frac{1}{2}tr(\boldsymbol{\mu} adj(\Lambdam))=-\frac{1}{2}\laO tr(\boldsymbol{\alpha})=\laO H$. By density of $\Sigma(\x)^c$ and the smoothness of $\KO$ and $H_\Omega$ we have the result for $\p\in \Sigma(\x)$.  
\end{proof}
\begin{prop}\label{zeros}
	Let $\x:U \to \R^3$ be a frontal and $\Omegam$ a tangent moving base of $\x$. The zeros of $\KO$ and $H_\Omega$ do not depend on the tangent moving base $\Omegam$ chosen for $\x$. Also, the signs are preserved if we restrict $\Omegam$ to compatibles tangent moving bases.
\end{prop}
\begin{proof}
	Let $\Omegam=\begin{pmatrix}\w_1&\w_2\end{pmatrix}$ and $\bar{\Omegam}=\begin{pmatrix}\bar{\w}_1&\bar{\w}_2\end{pmatrix}$ be tangent moving bases of $\x$, $\Lambdam=D\x^T\Omegam(\I_\Omega^T)^{-1}$ and $\bar{\Lambdam}=D\x^T\bar{\Omegam}(\I_{\bar{\Omega}}^T)^{-1}$. Since, $\spann{\w_1}{\w_2}=\n^\bot=\spann{\bar{\w}_1}{\bar{\w}_2}$, there exist $\mathbf{C}\in GL(2)$ such that $\Omegam=\bar{\Omegam}\mathbf{C}$. Then, $\Lambdam=D\x^T\bar{\Omegam}\mathbf{C}(\mathbf{C}^T\bar{\Omegam}^T\bar{\Omegam}\mathbf{C})^{-1}=D\x^T\bar{\Omegam}(\I_{\bar{\Omega}}^T)^{-1}(\mathbf{C}^T)^{-1}=\bar{\Lambdam}(\mathbf{C}^T)^{-1}$. On the other hand, $\mum=-\II_\Omega^T\I_\Omega^{-1}=-D\n^T\Omegam \I_\Omega^{-1}\\=-D\n^T\bar{\Omegam}\mathbf{C}(\mathbf{C}^T\bar{\Omegam}^T\bar{\Omegam}\mathbf{C})^{-1}=-\II_{\bar{\Omega}}^T\I_{\bar{\Omega}}^{-1}(\mathbf{C}^T)^{-1}=\bar{\mum}(\mathbf{C}^T)^{-1}$. Now, $K_{\bar{\Omega}}=det(\bar{\mum})=det(\mum)det(\mathbf{C})=det(\mathbf{C})\KO$ and $H_{\bar{\Omega}}=-\frac{1}{2}tr(\bar{\mum}adj(\bar{\Lambdam}))=-\frac{1}{2}tr(\mum\mathbf{C}^T adj(\mathbf{C}^T)adj(\Lambdam))\\=-\frac{1}{2}tr(\mum adj(\Lambdam))det(\mathbf{C})=det(\mathbf{C})H_{\Omega}$, then $\KO=0$ if and only if, $K_{\bar{\Omega}}=0$ and $H_{\Omega}=0$ if and only if, $H_{\bar{\Omega}}=0$. For the last assertion, observe that, if $\Omegam$ and $\bar{\Omegam}$ are compatibles, as $\Omegam=\bar{\Omegam}\mathbf{C}$, then $\w_{1}\times\w_{2}=det(\mathbf{C})\bar{\w_{1}}\times\bar{\w_{2}}$ and thus $det(\mathbf{C})=(\w_{1}\times\w_{2}\cdot\bar{\w_{1}}\times\bar{\w_{2}})|\bar{\w_{1}}\times\bar{\w_{2}}|^{-2}>0$, therefore $\KO$ and $H_{\Omega}$ have the same sign of $K_{\bar{\Omega}}$ and $H_{\bar{\Omega}}$.   
\end{proof}
If we have a frontal $\x:U \to \R^3$ with a tangent moving base $\Omegam$ and we compose $\x$ with a diffeomorphism $\mathbf{h}:V \to U$, this composition results a frontal ($D(\x\circ \mathbf{h})=(\Omegam\circ\mathbf{h}) (\Lambdam\circ\mathbf{h})^T D\mathbf{h})$ with  $\Omegam\circ\mathbf{h}$ being a tangent moving base of $\x\circ \mathbf{h}$. Similarly, if we compose $\x$ with a diffeomorphism $\mathbf{k}:W \to Z$, $\x(U)\subset W$, where $W$, $Z$ are open sets of $\R^2$, this composition results a frontal ($D(\mathbf{k}\circ \x)=D\mathbf{k}(\x)\Omegam \Lambdam^T)$ with $D\mathbf{k}(\x)\Omegam$ being a tangent moving base of $\x\circ \mathbf{h}$. Also, it is not difficult to see that if we have a front $\x:U \to \R^3$, then  $\x\circ\mathbf{h}$ and $\boldsymbol{\phi}\circ\x$ are fronts when  $\boldsymbol{\phi}:\R^3 \to \R^3$ is an isometry of $\R^3$ and $\mathbf{h}:V \to U$ is a diffeomorphism.   
\begin{prop}\label{change1}
	Let $\x:U \to \R^3$ be a frontal, $\mathbf{h}:V \to U$ a diffeomorphism, $\bar{\x}:=\x\circ \mathbf{h}$ the composite frontal, $\Omegam$ and $\Omegamb$ tangent moving bases of $\x$ and $\bar{\x}$ respectively. If $\KO$, $\HO$  are the relative curvatures for $\x$ and $\bar{K}_{\Omegab}$, $\bar{H}_{\Omegab}$ are the relative curvatures for $\bar{\x}$, then 
	\begin{itemize}
		\item $\KO(\mathbf{h}(x,y))=0$ if and only if $\bar{K}_{\Omegab}(x,y)=0$. 
		
		\item $\HO(\mathbf{h}(x,y))=0$ if and only if $\bar{H}_{\Omegab}(x,y)=0$.
		
	\end{itemize}    
\end{prop}
\begin{proof}
	For the first item, as $\hat{\Omegam}:=\Omegam(\mathbf{h})$ is a tangent moving base of $\bar{\x}(x,y)$, by (\ref{zeros}) $\bar{K}_{\hat{\Omega}}(x,y)=0$ if and only if $\bar{K}_{\Omegab}(x,y)=0$, but observe that $\bar{K}_{\hat{\Omega}}(x,y)=\KO(\mathbf{h}(x,y))$ which proves the item. On the other hand $\hat{\Lambdam}=D\bar{\x}^T\hat{\Omegam}(\I_{\hat{\Omega}})^{-1}=D\mathbf{h}^T\Lambdam(\mathbf{h})$ and $\hat{\n}=\n\circ \mathbf{h}$, then $\hat{\mum}=-\II_{\hat{\Omega}}^T\I_{\hat{\Omega}}^{-1}=-D\hat{\n}^T\hat{\Omegam} \I_{\hat{\Omega}}^{-1}=-D\mathbf{h}^TD\n^T(\mathbf{h})\hat{\Omegam} \I_{\hat{\Omega}}^{-1}=-D\mathbf{h}^T\II_{\Omega}^T(\mathbf{h})\I_{\Omega}^{-1}(\mathbf{h})=D\mathbf{h}^T\mum(\mathbf{h})$, thus $\bar{H}_{\hat{\Omega}}=-\frac{1}{2}tr(\hat{\mum} adj(\hat{\Lambdam}))=-\frac{1}{2}tr( adj(\hat{\Lambdam})\hat{\mum})=-det(D\mathbf{h})\frac{1}{2}tr(\mum(\mathbf{h}) adj(\Lambdam(\mathbf{h})))=det(D\mathbf{h})\HO(\mathbf{h})$ and therefore $\HO(\mathbf{h}(x,y))=0$ if and only if $\bar{H}_{\hat{\Omega}}(x,y)=0$. By proposition \ref{zeros} it follows the second item.      
\end{proof}
\begin{prop}\label{change2}
	Let $\x:U \to \R^3$ be a frontal, $\boldsymbol{\phi}:\R^3 \to \R^3$ an isometry of $\R^3$, $\bar{\x}:=\boldsymbol{\phi}\circ\x$ the composite frontal, $\Omegam$ and $\Omegamb$ tangent moving bases of $\x$ and $\bar{\x}$ respectively. If $\KO$, $\HO$  are the relative curvatures for $\x$ and $\bar{K}_{\Omegab}$, $\bar{H}_{\Omegab}$ are the relative curvatures for $\bar{\x}$, then 
	\begin{itemize}
		\item $\KO(u,v)=0$ if and only if $\bar{K}_{\Omegab}(u,v)=0$. 
		
		\item $\HO(u,v)=0$ if and only if $\bar{H}_{\Omegab}(u,v)=0$.
		
	\end{itemize}    
\end{prop}
\begin{proof}
	if $\boldsymbol{\phi}$ is an isometry, then we can write it in this form $\boldsymbol{\phi}(\p)=\mathbf{O}\p+\mathbf{a}$, where $\mathbf{O}\in \mathcal{M}_{3\times3}(\R)$ is an orthogonal matrix and $\mathbf{a}\in\R^3$ is a fixed vector. Thus, $\hat{\Omegam}:=\mathbf{O}\Omegam$ is a tangent moving base of $\bar{\x}$ and $\hat{\n}=\pm\mathbf{O}\n$ ($+$ if $det(\mathbf{O})=1$ and $-$ if $det(\mathbf{O})=-1$), then $\II_{\hat{\Omega}}=\pm(-\Omegam^T\mathbf{O}^T\mathbf{O}D\n)=\pm\II_{\Omega}$, $\I_{\hat{\Omega}}=\Omegam^T\mathbf{O}^T\mathbf{O}\Omegam=\I_{\Omega }$ and $\hat{\Lambdam}=\Lambdam$. Therefore, $\hat{\mum}=\pm\mum$ which implies $\KO=K_{\hat{\Omega}}$ and $\HO=\pm H_{\hat{\Omega}}$. By proposition \ref{zeros} it follows both items.       
\end{proof}
\begin{prop}\label{wfmatrix}
	Let $\x:U \to \R^3$ be a frontal and $\Omegam$ a tangent moving base of $\x$, then $\x$ is a front if and only if, 
	\begin{align}\label{wfm}
	\begin{pmatrix}
	\Lambdam^T\\
	\mum^T
	\end{pmatrix}
	\end{align} 
	has a $2\times2$ minor different of zero, for each $\p \in \Sigma(\x)$.
\end{prop}
\begin{proof}
	let $\n$ be the normal vector field along $\x$. By definition, $\x$ is a front if and only if,
	\begin{align}
	2=rank(\begin{pmatrix}
	D\x\\
	D\n
	\end{pmatrix})=rank(\begin{pmatrix}
	\Omegam\Lambdam^T\\
	\Omegam\mum^T
	\end{pmatrix})=rank(\begin{pmatrix}
	\Omegam&\mathbf{0}\\
	\mathbf{0}&\Omegam
	\end{pmatrix}\begin{pmatrix}
	\Lambdam^T\\
	\mum^T
	\end{pmatrix})=rank(\begin{pmatrix}
	\Lambdam^T\\
	\mum^T
	\end{pmatrix})
	\nonumber
	\end{align}
	which is equivalent to have a $2\times2$ minor of the matrix (\ref{wfm}) different of zero.
\end{proof}
The propositions \ref{zeros}, \ref{change1} and \ref{change2} now allow us to explore in which point any of $\KO$ and $\HO$ turns zero making change of coordinates, applying isometries of $\R^3$ and switching tangent moving bases. In the following theorem the necessary condition of the first item  was proved in (\cite{luc},Proposition 2.4) identifying $\HO$ as a $C^\infty$ extension of $\laO H$ for fronts with singular set having empty interior. The problem to use that result here is that the extension of $\laO H$ depends on the existence of regular points dense in the domain. However this holds in general for all fronts without that condition and the converse as well.
\begin{teo}\label{wft}
	Let $\x:U \to \R^3$ be a frontal, $\Omegam$ a tangent moving base of $\x$ and  $\p \in \Sigma(\x)$. Then,
	\begin{itemize}
		\item $\x:U \to \R^3$ is a front on a neighborhood $V$ of $\p$ with $rank(D\x(\p))=1$ if and only if $H_\Omega(\p)\neq 0$. 
		
		\item $\x:U \to \R^3$ is a front on a neighborhood $V$ of $\p$ with $rank(D\x(\p))=0$ if and only if $H_\Omega(\p)=0$ and $\KO(\p)\neq0$.
		
	\end{itemize}
\end{teo}
\begin{proof}
	For the first item, we can apply a change of coordinates $\mathbf{h}$ and an isometry $\boldsymbol{\phi}$ of $\R^3$ (making the line $D\x(\p)(\R^2)$ parallel to $(1,0,0)$) such that $\bar{\x}=\boldsymbol{\phi}\circ\x\circ\mathbf{h}=(u,b(u,v),c(u,v))$, $\mathbf{h}(0,0)=\p$, $b_u(0,0)=b_v(0,0)=c_u(0,0)=0$ and having a tangent moving base $\Omegamb$ in the form of remark \ref{specialbase}. Thus, $D\bar{\x}=\Omegamb\Lambdamb^T$, $\Lambdamb^T=D(u,b)$, $\bar{\mum}^T=D(-g_1det(\I_{\Omegab})^{-\frac{1}{2}},-g_2det(\I_{\Omegab})^{-\frac{1}{2}})$ and $(u,b)_u\cdot(g_1,g_2)_v=(u,b)_v\cdot(g_1,g_2)_u$ (by corollary \ref{D2c}). Hence, $c_u=g_1+g_2b_u$ and $g_{1v}+b_ug_{2v}=b_vg_{2u}$ which implies that $g_1(0,0)=g_{1v}(0,0)=0$. Since $\bar{\x}$ is wave front locally at $(0,0)$, by proposition \ref{wfmatrix} the matrix
	$$\begin{pmatrix}
	D(u,b)\\
	\bar{\mum}^T
	\end{pmatrix}$$
	has a minor $2\times2$ different of zero at $(0,0)$ and therefore $(-g_2det(\I_{\Omegab})^{-\frac{1}{2}})_v (0,0)\neq 0$. On the other hand a simple computation using the definition leads to $\bar{H}_{\Omegab}(0,0)=-\frac{1}{2}(-g_2det(\I_{\Omegab})^{-\frac{1}{2}})_v (0,0)\neq 0$, hence $\HO(\p)\neq 0$. Now, if we suppose that $H_\Omega(\p)\neq 0$, as $H_\Omega(\p)=-\frac{1}{2}(\lambda_{22}\mu_{11}-\lambda_{21}\mu_{12}+\lambda_{11}\mu_{22}-\lambda_{12}\mu_{21})(\p)$, then $(\lambda_{12}\mu_{21}-\lambda_{22}\mu_{11})(\p)\neq0$ or $(\lambda_{11}\mu_{22}-\lambda_{21}\mu_{12})(\p)\neq0$, which are two $2\times2$ minors of (\ref{wfm}) and also $\Lambdam(\p)\neq\mathbf{0}$. Thus, $rank(D\x(\p))=rank(\Lambdam(\p))=1$ and there exists a neighborhood $V$ of $\p$, where any of these two $2\times2$ minors is different of zero, therefore by proposition \ref{wft} $\x$ is a front on $V$. For the second item, if $\x$ is a front and $rank(D\x(\p))=rank(\Lambdam(\p))=0$, then $\Lambdam(\p)=\mathbf{0}$, $H_\Omega(\p)=0$ and by proposition \ref{wft} $\KO(\p)=det(\mum^T)\neq0$. Now, if $\KO(\p)\neq0$ and $H_\Omega(\p)=0$, there exist a neighborhood $V$ of $\p$ where $\KO\neq0$ and by proposition \ref{wft} $\x$ is a front on $V$. By the first item, $rank(D\x(\p))\neq1$ because $H_\Omega(\p)=0$, then $rank(D\x(\p))=0$.        
\end{proof}
From theorem \ref{wft} and proposition \ref{zeros} follows immediately the following corollary. 
\begin{coro}\label{wf}
	Let $\x:U \to \R^3$ be a frontal, this is a front if and only if, $(\KO,H_\Omega)\neq\mathbf{0}$ on $\Sigma(\x)$ for whatever tangent moving base $\Omegam$ of $\x$.
\end{coro}

\begin{ex}
	Let $\x:\R^2 \to \R^3$ defined by $\x(u,v):=(u^2,v^2,v^3+u^3)$, this is a frontal with $rank(D\x(\p))=0$ on $\p=(0,0)$ (Figure \ref{corank2b}).
	\begin{figure}[h]
		\begin{center}
			\includegraphics[scale=0.45]{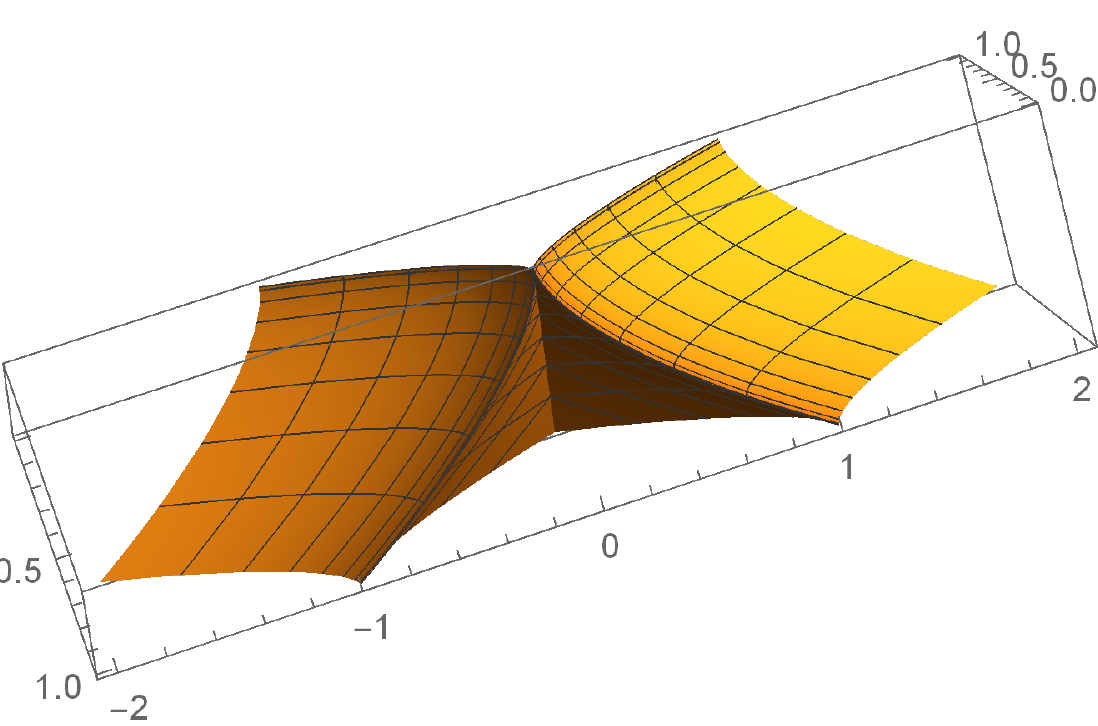}\qquad\qquad \includegraphics[scale=0.45]{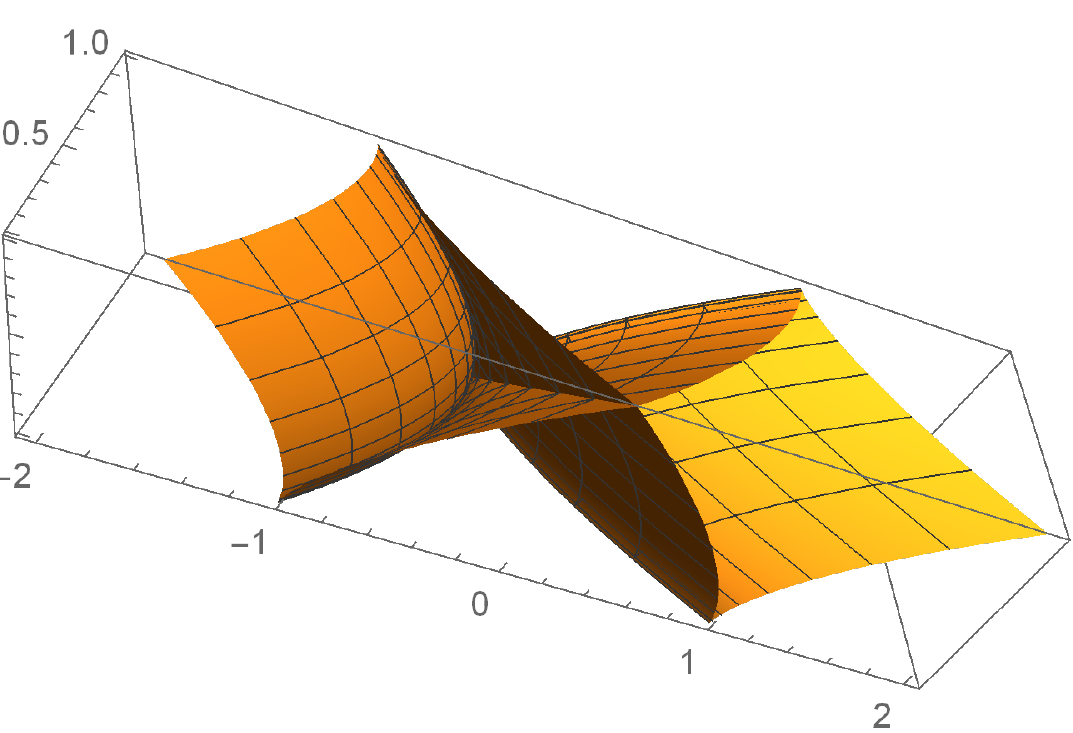}
		\end{center}
		\caption{A front with $rank(D\x(0,0))=0$.}\label{corank2b}
	\end{figure}
	We have the decomposition:
	\label{Corank2Db}
	\begin{align}
	D\x=\begin{pmatrix}
	2&0\\
	0&2\\
	3u&3v
	\end{pmatrix}\begin{pmatrix}
	u&0\\
	0&v
	\end{pmatrix}=\Omegam\Lambdam^T,\text{ where } \Omegam=\begin{pmatrix}
	2&0\\
	0&2\\
	3u&3v
	\end{pmatrix},\Lambdam=\begin{pmatrix}
	u&0\\
	0&v
	\end{pmatrix},\nonumber
	\end{align} being $\Omegam$ a tangent moving base of $\x$, then we have $\n=(-6u,-6v,4)\epsilon^{-\frac{1}{2}}$, $\w_{1u}=(0,0,3)$, $\w_{1v}=(0,0,0)$, $\w_{2u}=(0,0,0)$ and $\w_{2v}=(0,0,3)$. Thus 
	\begin{subequations}
		\begin{align}
		\I_\Omega=\begin{pmatrix}
		4+9u^2 & 9uv\\
		9uv & 4+9v^2
		\end{pmatrix}, \II_\Omega=\begin{pmatrix}
		12\epsilon^{-\frac{1}{2}} & 0\\
		0 & 12\epsilon^{-\frac{1}{2}}
		\end{pmatrix}\nonumber
		\end{align}
	\end{subequations}
	where $\epsilon=36u^2+36v^2+16$. Also, $\KO(u,v)=144(36u^2+36v^2+16)^{-2}\neq0$ and $H_\Omega(0,0)=0$, then by corollary \ref{wf}, $\x$ is a front.	
\end{ex}

\begin{ex}
	Let $\x:\R^2 \to \R^3$ defined by $\x(u,v):=(ue^u,v^2,(\frac{u^2}{2}+u)v^3)$, this is a frontal with $rank(D\x(\p))=0$ on $\p=(-1,0)$ (Figure \ref{corank2}). 
	\begin{figure}[h]
		\begin{center}
			\includegraphics[scale=0.31]{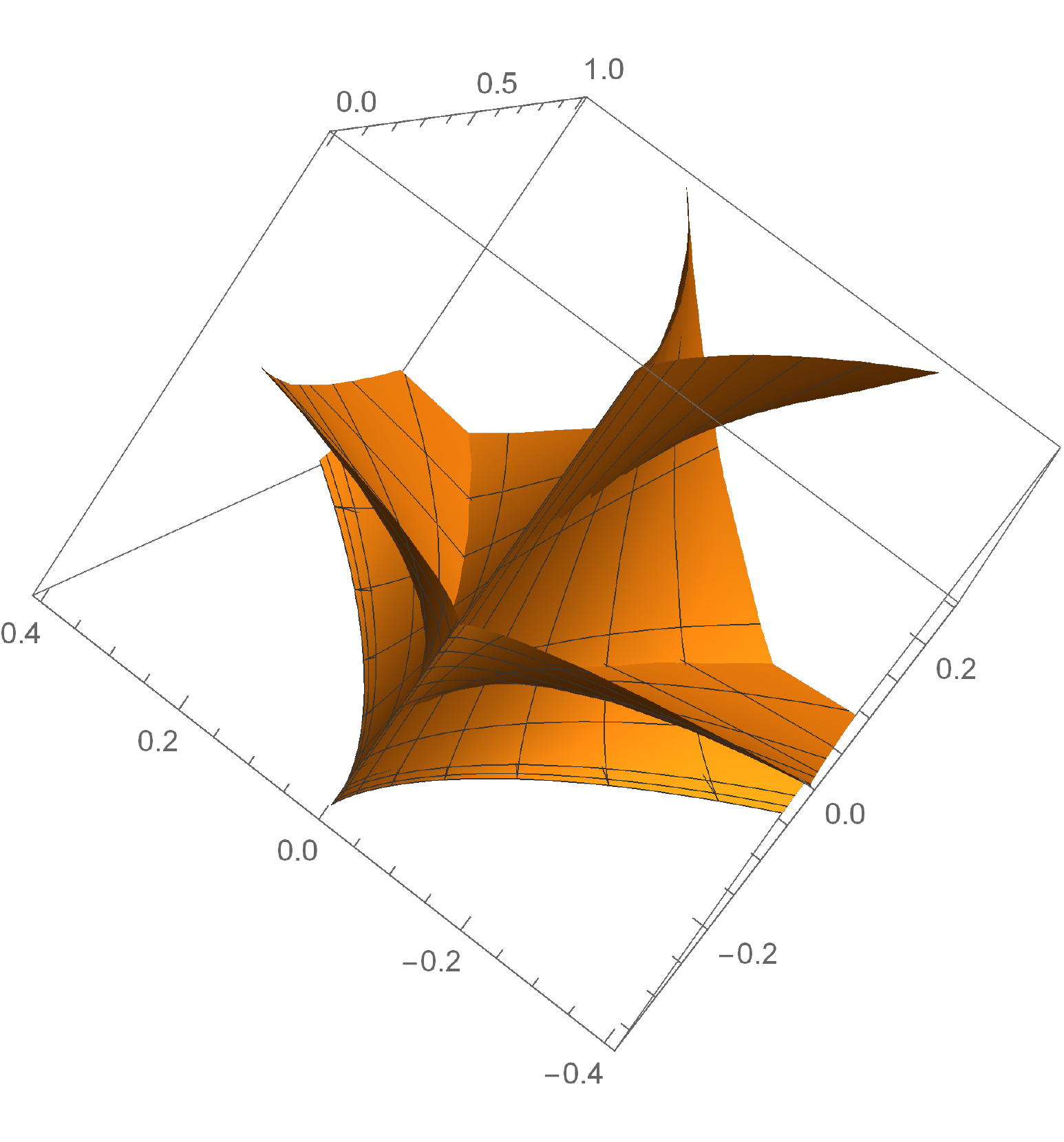}\qquad\qquad \includegraphics[scale=0.29]{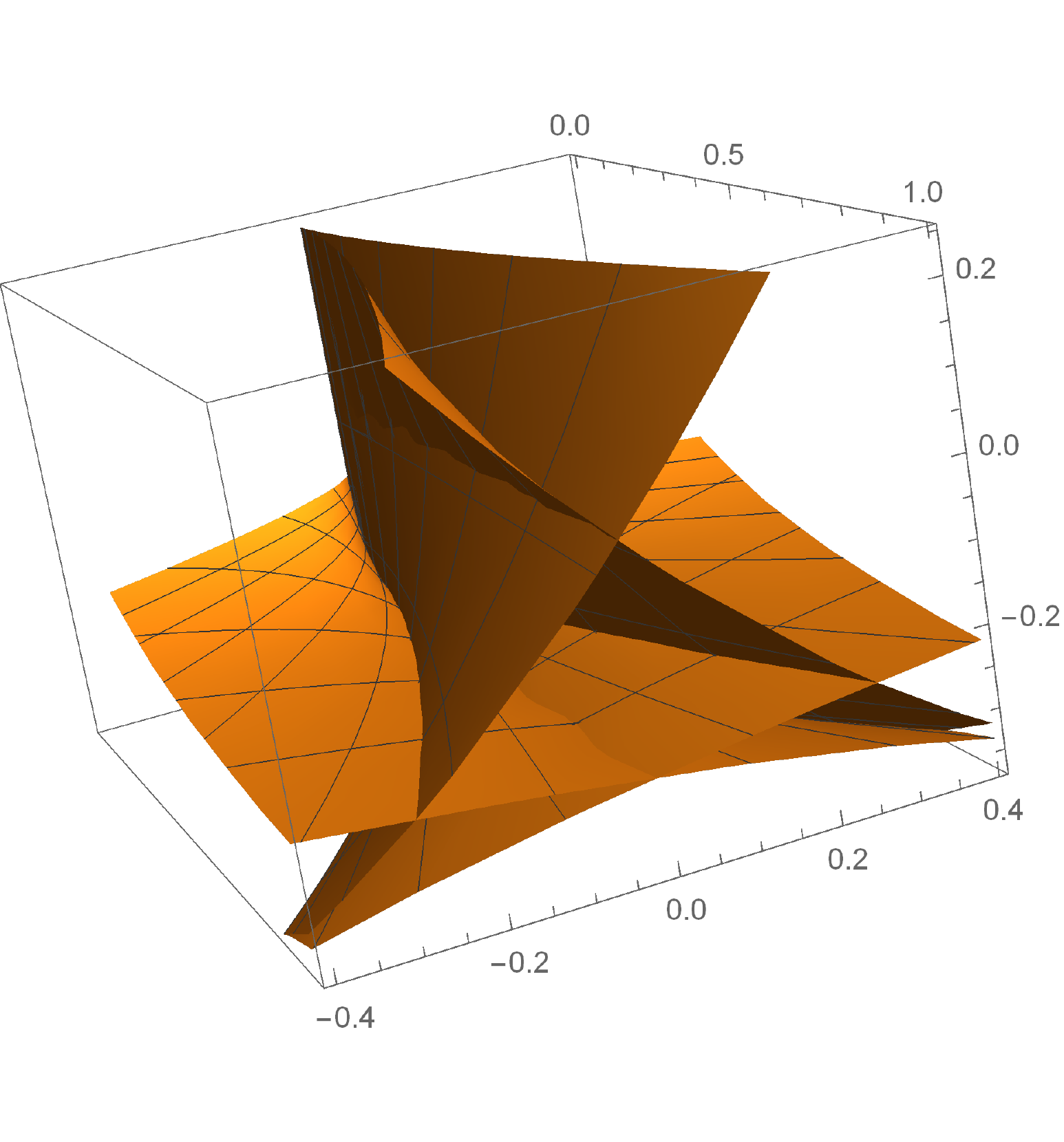}
		\end{center}
		\caption{A frontal with $rank(D\x(-1,0))=0$.}\label{corank2}
	\end{figure}
	We have the decomposition:
	\label{Corank2D}
	\begin{align}
	D\x=\begin{pmatrix}
	e^u&0\\
	0&2\\
	v^3&(\frac{u^2}{2}+u)3v
	\end{pmatrix}\begin{pmatrix}
	1+u&0\\
	0&v
	\end{pmatrix}=\Omegam\Lambdam^T,\nonumber\\
	\text{ where } \Omegam=\begin{pmatrix}
	e^u&0\\
	0&2\\
	v^3&(\frac{u^2}{2}+u)3v
	\end{pmatrix}, \Lambdam=\begin{pmatrix}
	1+u&0\\
	0&v
	\end{pmatrix},\nonumber
    \end{align}
    being $\Omegam$ a tangent moving base of $\x$, then we have
	$\n=(-2v^3,-e^u(\frac{u^2}{2}+u)3v,2e^u)\delta^{-\frac{1}{2}}$, $\w_{1u}=(e^u,0,0)$, $\w_{1v}=(0,0,3v^2)$, $\w_{2u}=(0,0,(u+1)3v)$ and $\w_{2v}=(0,0,3(\frac{u^2}{2}+u))$. Thus  
	\begin{subequations}
		\begin{align}
		\I_\Omega=\begin{pmatrix}
		e^{2u}+v^6 & 3(\frac{u^2}{2}+u)v^4\\
		3(\frac{u^2}{2}+u)v^4 & 4+9(\frac{u^2}{2}+u)^2v^2
		\end{pmatrix}, 
		\II_\Omega=\begin{pmatrix}
		-2v^3e^u & 6e^uv^2\\
		6(1+u)e^uv & 6e^u(\frac{u^2}{2}+u)
		\end{pmatrix}\delta^{-\frac{1}{2}}\nonumber
		\end{align}
	\end{subequations}
	where $\delta=4v^6+e^{2u}(9(\frac{u^2}{2}+u)^2v^2+4)$. Also, $\KO(-1,0)=0$ and $H_\Omega(-1,0)=0$, then by corollary \ref{wf}, $\x$ is not a front.
\end{ex}
\section{The Compatibility Equations}

Let $\Omegam=\begin{pmatrix}\w_1&\w_2\end{pmatrix}:U \to \mathcal{M}_{3\times2}(\R)$ be a moving base and $\n=\frac{\w_1\times\w_2}{\|\w_1\times\w_2\|}$. We have that ${\w_1, \w_2, \n}$ is a base of $\R^3$, then there are real functions $(p_{ij})$ and $(q_{ij})$  defined in $U$, $i \in \{1, 2, 3\}$ such that:
\begin{subequations}
	\begin{align}
	\w_{1u}&=p_{11}\w_1+p_{12}\w_2+p_{13}\n\label{co1}
	\\
	\w_{2u}&=p_{21}\w_1+p_{22}\w_2+p_{23}\n\label{co2}
	\\
	\n_{u}&=p_{31}\w_1+p_{32}\w_2+p_{33}\n\label{co3}
	\\
	\w_{1v}&=q_{11}\w_1+q_{12}\w_2+q_{13}\n\label{co4}
	\\
	\w_{2v}&=q_{21}\w_1+q_{22}\w_2+q_{23}\n\label{co5}
	\\
	\n_{v}&=q_{31}\w_1+q_{32}\w_2+q_{33}\n\label{co6}
	\end{align}	
\end{subequations}
If we set the matrix $\Wm:=\begin{pmatrix}\w_1&\w_2&\n\end{pmatrix}\in GL(3)$ whose columns are $\w_{1}$, $\w_{2}$ and $\n$. Also, denoting by $\P:=(p_{ij})$ and $\Q:=(q_{ij})$, we have:  
\begin{subequations}
\label{Wd}
\begin{align}
\Wm_u=\Wm\P^T\\ 
\Wm_v=\Wm\Q^T
\end{align}
\end{subequations}
which is equivalent to:
\begin{subequations}
	\label{SP2}
	\begin{align}
	\Wm^T_u=\P\Wm^T\label{SP2a}\\ 
	\Wm^T_v=\Q\Wm^T\label{SP2b}
	\end{align}
\end{subequations}
then, we have that $\P=\Wm^T_u(\Wm^T)^{-1}=\Wm^T_u\Wm\Wm^{-1}(\Wm^T)^{-1}=\Wm^T_u\Wm(\Wm^T\Wm)^{-1}$ and $\Q=\Wm^T_v(\Wm^T)^{-1}=\Wm^T_v\Wm\Wm^{-1}(\Wm^T)^{-1}=\Wm^T_v\Wm(\Wm^T\Wm)^{-1}$. Considering $\Wm=\begin{pmatrix}\Omegam&\n \end{pmatrix}$ as a block matrix, we have:
\begin{align}
\P&=\Wm^T_u\Wm(\Wm^T\Wm)^{-1}=\begin{pmatrix}\Omegam_u^T\\
\n_u^T\end{pmatrix}\begin{pmatrix}\Omegam&\n \end{pmatrix}(\begin{pmatrix}\Omegam^T\\
\n^T\end{pmatrix}\begin{pmatrix}\Omegam&\n \end{pmatrix})^{-1}\\
&=\begin{pmatrix}\Omegam_u^T\Omegam&\Omegam_u^T\n\\
\n_u^T\Omegam&\n_u^T\n \end{pmatrix}\begin{pmatrix}\Omegam^T\Omegam&\Omegam^T\n\\
\n^T\Omegam&\n^T\n \end{pmatrix}^{-1}=\begin{pmatrix}\Omegam_u^T\Omegam&\Omegam_u^T\n\\
\n_u^T\Omegam&0 \end{pmatrix}\begin{pmatrix}\I_\Omega&\mathbf{0}\\
\mathbf{0}&1 \end{pmatrix}^{-1}\nonumber\\
&=\begin{pmatrix}\Omegam_u^T\Omegam&\Omegam_u^T\n\\
\n_u^T\Omegam&0 \end{pmatrix}\begin{pmatrix}\I_\Omega^{-1}&\mathbf{0}\\
\mathbf{0}&1 \end{pmatrix}=\begin{pmatrix}\Omegam_u^T\Omegam\I_\Omega^{-1}&\Omegam_u^T\n\\
\n_u^T\Omegam\I_\Omega^{-1}&0 \end{pmatrix}\nonumber
\end{align}
from (\ref{WO}), we have $\n_u^T=\mum_{(1)}^T\Omegam^T$ and $\n_v^T=\mum_{(2)}^T\Omegam^T$. Then,
\begin{align}
\P=\begin{pmatrix}\Omegam_u^T\Omegam\I_\Omega^{-1}&\Omegam_u^T\n\\
\n_u^T\Omegam\I_\Omega^{-1}&0 \end{pmatrix}=\begin{pmatrix}\TA{1}&\Omegam_u^T\n\\
\mum_{(1)}^T\Omegam^T\Omegam\I_\Omega^{-1}&0 \end{pmatrix}=\begin{pmatrix}\TA{1}&\Omegam_u^T\n\\
\mum_{(1)}^T&0 \end{pmatrix}
\end{align}
Finally, using (\ref{SFO}) and by analogy with the same procedure for $\Q$, we get:
\begin{align}
\P=\begin{pmatrix}
\Ta{1}{1}{1} & \Ta{1}{2}{1} & \eo \\	
\Ta{2}{1}{1} & \Ta{2}{2}{1} & \fdo \\
\mu_{11} & \mu_{12} & 0
\end{pmatrix}\label{P}\\ 
\Q=\begin{pmatrix}
\Ta{1}{1}{2} & \Ta{1}{2}{2} & \fuo \\
\Ta{2}{1}{2} & \Ta{2}{2}{2} & \go \\
\mu_{21}     &   \mu_{22}   & 0\end{pmatrix}\label{Q}
\end{align}
now, as $\Wm^T_{uv}=\Wm^T_{vu}$, then $\P_v\Wm^T+\P\Wm_v^T=\Q_u\Wm^T+\Q\Wm_u^T$. Using (\ref{SP2a}) and (\ref{SP2b}) in the last equality, $\P_v\Wm^T+\P\Q\Wm^T=\Q_u\Wm^T+\Q\P\Wm^T$, then $(\P_v-\Q_u+\P\Q-\Q\P)\Wm^T=0$ and finally we get:  
\begin{align}
\P_v-\Q_u + [\P,\Q]=0\label{comp2}
\end{align}
which is the compatibility condition of the system (\ref{SP2}) by corollary \ref{frobl}. 

Using (\ref{P}) and (\ref{Q}) to compute each component $(i,j)$ of (\ref{comp2}) we obtain the following equations that we call the {\it $\Omega$-relative compatibility equations (RCE)}:
\begin{subequations}
	\label{estructure}
	\begin{align}
	&(1,1)  \quad (\Ta{1}{1}{1})_v-(\Ta{1}{1}{2})_u=\Ta{1}{1}{2}\Ta{1}{1}{1}-\Ta{1}{1}{1}\Ta{1}{1}{2}+\Ta{1}{2}{2}\Ta{2}{1}{1}-\Ta{2}{1}{2}\Ta{1}{2}{1}+\mu_{11}\fuo-\mu_{21}\eo\label{c1}\\
	&(1,2) \quad (\Ta{1}{2}{1})_v-(\Ta{1}{2}{2})_u=\Ta{1}{1}{2}\Ta{1}{2}{1}+\Ta{1}{2}{2}\Ta{2}{2}{1}-\Ta{1}{1}{1}\Ta{1}{2}{2}-\Ta{1}{2}{1}\Ta{2}{2}{2}+\mu_{12}\fuo-\mu_{22}\eo\label{c2}\\
	&(2,1) \quad (\Ta{2}{1}{1})_v-(\Ta{2}{1}{2})_u=\Ta{2}{1}{2}\Ta{1}{1}{1}+\Ta{2}{2}{2}\Ta{2}{1}{1}-\Ta{2}{1}{1}\Ta{1}{1}{2}-\Ta{2}{2}{1}\Ta{2}{1}{2}+\mu_{11}\go-\mu_{21}\fdo\label{c3}\\
	&(2,2) \quad (\Ta{2}{2}{1})_v-(\Ta{2}{2}{2})_u=\Ta{2}{2}{2}\Ta{2}{2}{1}-\Ta{2}{2}{1}\Ta{2}{2}{2}+\Ta{1}{2}{1}\Ta{2}{1}{2}-\Ta{1}{2}{2}\Ta{2}{1}{1}+\mu_{12}\go-\mu_{22}\fdo\label{c4}\\
	&(1,3) \quad \mu_{11v}-\mu_{21u}=\Ta{1}{1}{1}\mu_{21}+\Ta{2}{1}{1}\mu_{22}-\Ta{1}{1}{2}\mu_{11}-\Ta{2}{1}{2}\mu_{12}\label{c5}\\
	&(2,3) \quad \mu_{12v}-\mu_{22u}=\Ta{1}{2}{1}\mu_{21}+\Ta{2}{2}{1}\mu_{22}-\Ta{1}{2}{2}\mu_{11}-\Ta{2}{2}{2}\mu_{12}\label{c6}\\ 
	&(3,1) \quad (\eo)_v-(\fuo)_u=\eo\Ta{1}{1}{2}+\fdo\Ta{1}{2}{2}-\fuo\Ta{1}{1}{1}-\go\Ta{1}{2}{1}\label{c7}\\
	&(3,2) \quad (\fdo)_v-(\go)_u=\eo\Ta{2}{1}{2}+\fdo\Ta{2}{2}{2}-\fuo\Ta{2}{1}{1}-\go\Ta{2}{2}{1}\label{c8}\\
	&(3,3) \quad \eo\mu_{21}+\fdo\mu_{22}-\fuo\mu_{11}-\go\mu_{12}=0\label{c9}  
	\end{align}
\end{subequations}
Using that the $\Omega$-relative curvature $\KO=det(\boldsymbol{\mu})=\frac{det(\II_\Omega)}{det(\I_\Omega)}=\frac{\eo\go-\fuo\fdo}{\EO\GO-\FO^2}$ and $\boldsymbol{\mu}=-\II_\Omega^T\I_\Omega^{-1}$ in (\ref{c2}) we get: 
\begin{align}
(\Ta{1}{2}{2})_u-(\Ta{1}{2}{1})_v+\Ta{1}{1}{2}\Ta{1}{2}{1}+\Ta{1}{2}{2}\Ta{2}{2}{1}-\Ta{1}{2}{1}\Ta{2}{2}{2}-\Ta{1}{1}{1}\Ta{1}{2}{2}=-\EO \KO\label{GaussT}.
\end{align}
On the other hand, let $\x:U \to \R^3$ be a frontal and $\Omegam=\begin{pmatrix}\w_1&\w_2\end{pmatrix}$ a tangent moving base of $\x$. Then, $D\x=\Omegam\Lambdam^T$ and we have that,
\begin{subequations}
	\begin{align}
	\x_u=\lambda_{11}\w_1+\lambda_{12}\w_2\\
	\x_v=\lambda_{21}\w_1+\lambda_{22}\w_2
	\label{DOS}
	\end{align}
\end{subequations}
where $\Lambdam=(\lambda_{ij})$. Setting,
$$\Lambdamb:=\begin{pmatrix}
\lambda_{11} & \lambda_{12} & 0 \\
\lambda_{21} & \lambda_{22} & 0 \\
0 & 0 & 1\end{pmatrix}, \ \X:=\begin{pmatrix}
D\x&\n\end{pmatrix},$$ 
we have $\X=\Wm\Lambdamb^T$, where $\Wm=\begin{pmatrix}\Omegam&\n\end{pmatrix}$. Denoting $\im,\jm,\km$ the canonical base of $\R^3$, the compatibility condition $\x_{uv}=\x_{vu}$ is equivalent to
\begin{align}
	\X_u\jm=\X_v\im\label{C1},
\end{align}
using (\ref{Wd}) with (\ref{C1}) we have
\begin{align}
\Wm\P^T\Lambdamb^T\jm+\Wm\Lambdamb_u^T\jm=\Wm_u\Lambdamb^T\jm+\Wm\Lambdamb_u^T\jm=\Wm_v\Lambdamb^T\im+\Wm\Lambdamb_v^T\im=\Wm\Q^T\Lambdamb^T\im+\Wm\Lambdamb_v^T\im\label{C2},
\end{align} 
then (\ref{C1}) is equivalent to
\begin{align}
\P^T\Lambdamb^T\jm+\Lambdamb_u^T\jm=\Q^T\Lambdamb^T\im+\Lambdamb_v^T\im\label{C3}.
\end{align} 
Computing each component of \ref{C3}, we get the following equations that we call {\it singular compatibility equations (SCE)}:
\begin{subequations}
	\label{CS}
	\begin{align}
	&\lambda_{11v}-\lambda_{21u}=\Ta{1}{1}{1}\lambda_{21}+\Ta{2}{1}{1}\lambda_{22}-\Ta{1}{1}{2}\lambda_{11}-\Ta{2}{1}{2}\lambda_{12}\label{CS1}\\
	&\lambda_{12v}-\lambda_{22u}=\Ta{1}{2}{1}\lambda_{21}+\Ta{2}{2}{1}\lambda_{22}-\Ta{1}{2}{2}\lambda_{11}-\Ta{2}{2}{2}\lambda_{12}\label{CS2}\\
	&\lambda_{11}\fuo+\lambda_{12}\go=\lambda_{21}\eo+\lambda_{22}\fdo\label{CS3}
	\end{align}
\end{subequations}
If we set the matrices:
$$\eu:=\begin{pmatrix}
1\\
0
\end{pmatrix}, \ed:=\begin{pmatrix}
0\\
1
\end{pmatrix},$$
we can write all these equations with a very useful compact notation.\\ 
Equations (\ref{CS1}) and (\ref{CS2}): 
\begin{align}
\ed^T(\Lambdam\TA{1}+\Lambdam_u)=\eu^T(\Lambdam\TA{2}+\Lambdam_v)\label{cc1}.
\end{align} 
Equation (\ref{CS3}): 
\begin{align}
\Lambdam_{(1)}\II_\Omega^{(2)}=\Lambdam_{(2)}\II_\Omega^{(1)}\label{cc2},
\end{align} 
that is, $\Lambdam\II_\Omega$ is symmetric.\\
Equations (\ref{c1}), (\ref{c2}), (\ref{c3}) and (\ref{c4}):
\begin{align}
\TA{1v}-\TA{2u}+\TA{1}\TA{2}-\TA{2}\TA{1}+\II_\Omega^{(1)}\ed^T\mum-\II_\Omega^{(2)}\eu^T\mum=0\label{cc3}.
\end{align} 
Equations (\ref{c5}) and (\ref{c6}):
\begin{align}
\ed^T(\mum\TA{1}+\mum_u)=\eu^T(\mum\TA{2}+\mum_v)\label{cc4}.
\end{align}
Equations (\ref{c7}) and (\ref{c8}):
\begin{align}
\ed^T(\II_\Omega^T\TA{1}^T-\II_{\Omega u})=\eu^T(\II_\Omega^T\TA{2}^T-\II_{\Omega v})\label{cc5}.
\end{align}
Equation (\ref{c9}): 
\begin{align}
\mum_{(1)}\II_\Omega^{(2)}=\mum_{(2)}\II_\Omega^{(1)}\label{cc6},
\end{align} 
that is, $\mum\II_\Omega$ is symmetric.
\begin{prop}
Let $\I_\Omega, \II_\Omega, \TA{1}, \TA{2}:U \to \mathcal{M}_{2\times2}(\R)$ be arbitrary smooth maps and $\I_\Omega$ symmetric positive definite. If we set $\mum=-\II_\Omega^T\I_\Omega^{-1}$, $\KO=det(\mum)$ and we have that
\begin{align}
\I_\Omega\TA{1}^T+\TA{1}\I_\Omega=\I_{\Omega u},\label{gu}\\
\I_\Omega\TA{2}^T+\TA{2}\I_\Omega=\I_{\Omega v},\label{gv}
\end{align}
 denoting 
 $$\I_\Omega=\begin{pmatrix}
 \EO&\FO\\
 \FO&\GO
 \end{pmatrix},\II_\Omega=\begin{pmatrix}
 \eo&\fuo\\
 \fdo&\go
 \end{pmatrix},$$
 then the equation (\ref{cc3}) is satisfied if and only if, the equation (\ref{GaussT}) is satisfied. 
\end{prop}
\begin{proof}
	The equation (\ref{cc3}) is satisfied if and only if, the resulting equation of multiplying this by the right side with $\I_\Omega$ is satisfied
	\begin{align}
	\TA{1v}\I_\Omega-\TA{2u}\I_\Omega+\TA{1}\TA{2}\I_\Omega-\TA{2}\TA{1}\I_\Omega-\II_\Omega^{(1)}\II_\Omega^{(2)T}+\II_\Omega^{(2)}\II_\Omega^{(1)T}=0.\label{ge}
	\end{align}
	observe that $-\II_\Omega^{(1)}\II_\Omega^{(2)T}+\II_\Omega^{(2)}\II_\Omega^{(1)T}$ is skew-symmetric. Let us set,
	\begin{align}
	\mathbf{A}:=\TA{1v}\I_\Omega-\TA{2u}\I_\Omega+\TA{1}\TA{2}\I_\Omega-\TA{2}\TA{1}\I_\Omega,\label{ga}\\ \mathbf{B}=\begin{pmatrix}
	0 &-b\\
	b & 0
	\end{pmatrix}:=-\II_\Omega^{(1)}\II_\Omega^{(2)T}+\II_\Omega^{(2)}\II_\Omega^{(1)T},
	\end{align}
	we have that
	\begin{align}
	-\mathbf{A}^T=-\I_\Omega\TA{1v}^T+\I_\Omega\TA{2u}^T-\I_\Omega\TA{2}^T\TA{1}^T+\I_\Omega\TA{1}^T\TA{2}^T.\label{gat}
	\end{align}
	On the other hand, deriving (\ref{gu}) by $v$, (\ref{gv}) by $u$ and subtracting the results we get:
	\begin{align}
	\I_\Omega\TA{2u}^T-\I_\Omega\TA{1v}^T-\TA{1}\I_{\Omega v}+\TA{2}\I_{\Omega u}=\TA{1v}\I_\Omega-\TA{2u}\I_\Omega-\I_{\Omega u}\TA{2}^T+\I_{\Omega v}\TA{1}^T.\label{gs}
	\end{align}
	Substituting in (\ref{gs}), $\I_{\Omega u}$ and $\I_{\Omega u}$ by (\ref{gu}) and (\ref{gv}), we obtain canceling similar terms that, the right side of (\ref{ga}) is equal to the right side of (\ref{gat}). Then, $\mathbf{A}$ is skew-symmetric having the form
	$$\mathbf{A}=\begin{pmatrix}
	0 &-a\\
	a & 0
	\end{pmatrix},$$
	hence (\ref{ge}) is satisfied if and only if, $a+b$=0 as also, (\ref{cc3}) can be expressed in this form,
	\begin{align}
	\begin{pmatrix}
	0 &-a-b\\
	a+b & 0
	\end{pmatrix}\I_\Omega^{-1}=\mathbf{0}.\label{gf}
	\end{align} Computing the component $(1,2)$ of (\ref{gf}) we have $\EO(-a-b)\det(\I_\Omega)^{-1}=0$, then $a+b=0$ if and only if, the component $(1,2)$ of (\ref{cc3}) is satisfied, which is the equation (\ref{c2}) that is simplified to (\ref{GaussT}).         
\end{proof}
\begin{prop}
	Let $\I_\Omega, \II_\Omega, \TA{1}, \TA{2}:U \to \mathcal{M}_{2\times2}(\R)$ be arbitrary smooth maps and $\I_\Omega$ symmetric non-singular. If we set $\mum:=-\II_\Omega^T\I_\Omega^{-1}$ and we have that
	\begin{align}
	\I_\Omega\TA{1}^T+\TA{1}\I_\Omega=\I_{\Omega u},\label{equiv1}\\
	\I_\Omega\TA{2}^T+\TA{2}\I_\Omega=\I_{\Omega v},\label{equiv2}
	\end{align}
	then, the equation (\ref{cc5}) is satisfied if and only if, the equation (\ref{cc4}) is satisfied.
\end{prop}
\begin{proof}
	Using that $\mum=-\II_\Omega^T\I_\Omega^{-1}$, we substitute $\mum$, $\mum_u$ and $\mum_v$ in (\ref{cc4}), we get $$\ed^T(-\II_\Omega^T\I_\Omega^{-1}\TA{1}-\II_{\Omega u}^T\I_{\Omega}^{-1}-\II_\Omega^T(\I_{\Omega }^{-1})_u)=\eu^T(-\II_\Omega^T\I_\Omega^{-1}\TA{2}-\II_{\Omega v}^T\I_{\Omega}^{-1}-\II_\Omega^T(\I_{\Omega }^{-1})_v)$$
	that is satisfied if and only if, the resulting equation of multiply this by the right side with $\I_\Omega$ is satisfied, in which we can later substitute $(\I_{\Omega}^{-1})_u \I_\Omega=-\I_\Omega^{-1}\I_{\Omega u}$, $(\I_{\Omega}^{-1})_v \I_\Omega=-\I_\Omega^{-1}\I_{\Omega v}$, factorize similar terms in both sides and get
	$$\ed^T(\II_\Omega^T\I_\Omega^{-1}(\I_{\Omega u}-\TA{1}\I_\Omega)-\II_{\Omega u}^T)=\eu^T(\II_\Omega^T\I_\Omega^{-1}(\I_{\Omega v}-\TA{2}\I_\Omega)-\II_{\Omega v}^T).$$
	Since $\I_{\Omega u}-\TA{1}\I_\Omega=\I_\Omega\TA{1}^T$ and $\I_{\Omega v}-\TA{2}\I_\Omega=\I_\Omega\TA{2}^T$ by hypothesis, subtituting these, the equation becomes in (\ref{cc5}). 
\end{proof}
\begin{rem}
	Since that equation (\ref{cc6}) is always satisfied by definition of $\mum$ and as every frontal satisfy (\ref{equiv1}) and (\ref{equiv2}) (proposition \ref{propE}), by these two last proposition (RCE) are equivalent to (\ref{GaussT}), (\ref{c7}) and (\ref{c8}).  
\end{rem}
\section{The Fundamental Theorem}

\begin{teo}\label{TF}
	
	Let $E,F,G,e,f,g$ smooth functions defined in an open set $U\subset \R^2$, with $E\geq0$, $G\geq0$ and $EG-F^2\geq0$. Assume that the given functions have the following decomposition:
	
	\begin{subequations}
		\label{FFS}
		\begin{align}
		\begin{pmatrix}
		E&F\\
		F&G
		\end{pmatrix}&=\begin{pmatrix}
		\la{1}{1} & \la{1}{2}\\
		\la{2}{1} & \la{2}{2}
		\end{pmatrix}\begin{pmatrix}
		\EO&\FO\\
		\FO&\GO
		\end{pmatrix}\begin{pmatrix}
		\la{1}{1} & \la{1}{2}\\
		\la{2}{1} & \la{2}{2}
		\end{pmatrix}^T\label{FFS1}
		\\ 
		\begin{pmatrix}
		e&f\\
		f&g
		\end{pmatrix}&=\begin{pmatrix}
		\la{1}{1} & \la{1}{2}\\
		\la{2}{1} & \la{2}{2}
		\end{pmatrix}\begin{pmatrix}
		\eo&\fuo\\
		\fdo&\go
		\end{pmatrix}\label{FFS2}
		\end{align}
	\end{subequations}	
	in which all the components are smooth real functions defined in $U$, $\EO>0$, $\GO>0$, $\EO\GO-\FO^2>0$, $\laO^{-1}(0)$ has empty interior and
	\begin{subequations}
		\begin{align}
		\Lambdam_{(1)u}^{\phantomsection}\begin{pmatrix}
		\EO&\FO\\
		\FO&\GO
		\end{pmatrix}\Lambdam_{(2)}^T-\Lambdam_{(1)}^{\phantomsection}\begin{pmatrix}
		\EO&\FO\\
		\FO&\GO
		\end{pmatrix}\Lambdam_{(2)u}^T+E_v-F_u \in \mathfrak{T}_\Omega\\
		\Lambdam_{(1)v}^{\phantomsection}\begin{pmatrix}
		\EO&\FO\\
		\FO&\GO
		\end{pmatrix}\Lambdam_{(2)}^T-\Lambdam_{(1)}^{\phantomsection}\begin{pmatrix}
		\EO&\FO\\
		\FO&\GO
		\end{pmatrix}\Lambdam_{(2)v}^T+F_v-G_u \in \mathfrak{T}_\Omega,
		\end{align}
	\end{subequations}
where $\Lambdam=(\lambda_{ij})$, $\laO=det(\Lambdam)$ and $\mathfrak{T}_\Omega$ is the principal ideal generated by $\laO$ in the ring $C^\infty(U,\R)$.
$E,F,G,e,f,g$ formally satisfy the Gauss and Mainardi-Codazzi equations for all $(u,v)\in U-\laO^{-1}(0)$. Then, for each $(u_0,v_0)\in U$ there exists a neighborhood $V\subset U$ of $(u_0,v_0)$ and a frontal $\x:V \to \x(V)\subset \R^3$ with a tangent moving base $\Omegam$ such that $D\x=\Omegam\Lambdam^T$, $$\I_\Omega=\begin{pmatrix}
\EO&\FO\\
\FO&\GO
\end{pmatrix}, \ \II_\Omega=\begin{pmatrix}
\eo&\fuo\\
\fdo&\go
\end{pmatrix}$$
and the frontal $\x$ has $E,F,G$ and $e,f,g$ as coefficients of the first and second fundamental forms, respectively. Furthermore, if $U$ is connected and if 
$${\mathbf{\bar{\x}}}:U \to \R^3\text{ and } \bar{\Omegam}:U \to \R^3$$
are another frontal and a tangent moving base satisfying the same conditions, then there exist a translation $\mathbf{T}$ and a proper linear orthogonal transformation $\rhom$ in $\R^3$ such that $\Omegamb=\rhom\Omegam$ and ${\mathbf{\bar{\x}}}=\mathbf{T}\circ \rhom \circ \x$.    
\end{teo}
\begin{lem}\label{l1}
	If we have:
	\begin{subequations}
		\begin{align}
		\GAb{1}\Lambdamb-\Lambdamb_u=\Lambdamb\TAb{1}\label{f1}\\
		\GAb{2}\Lambdamb-\Lambdamb_v=\Lambdamb\TAb{2}\label{f2}
		\end{align}
	\end{subequations}
in which $\Lambdamb,\TAb{1},\TAb{2}:U \to \mathcal{M}_{n\times n}(\R)$ and $\GAb{1},\GAb{2}:U-det^{-1}(\Lambdamb)(0) \to \mathcal{M}_{n\times n}(\R)$ are smooth maps with $int(det^{-1}(\Lambdamb)(0))=\emptyset$. Then,\\
	
	$\GAb{1v}-\GAb{2u}+[\GAb{1},\GAb{2}]=0$ is equivalent to $\TAb{1v}-\TAb{2u}+[\TAb{1},\TAb{2}]=0$ in $U$. 
\end{lem}
\begin{proof}
	Deriving (\ref{f1}) in $v$, (\ref{f2}) in $u$ we get:
\begin{subequations}
		\begin{align}
		\Lambdamb_v\TAb{1}+\Lambdamb\TAb{1v}=\GAb{1v}\Lambdamb+\GAb{1}\Lambdamb_v-\Lambdamb_{uv}\label{f1,1}\\
		\Lambdamb_u\TAb{2}+\Lambdamb\TAb{2u}=\GAb{2u}\Lambdamb+\GAb{2}\Lambdamb_u-\Lambdamb_{vu}\label{f2,1}
		\end{align}
	\end{subequations}
	Subtracting (\ref{f2,1}) from (\ref{f1,1})
	\begin{subequations}
		\begin{align}
		\Lambdamb(\TAb{1v}-\TAb{2u})+\Lambdamb_v\TAb{1}-\Lambdamb_u\TAb{2}=(\GAb{1v}-\GAb{2u})\Lambdamb+\GAb{1}\Lambdamb_v-\GAb{2}\Lambdamb_u\label{f3}
		\end{align}
	\end{subequations}
Substituting (\ref{f1}) and (\ref{f2}) in (\ref{f3}) on right side  
\begin{subequations}
	\begin{align}
	\Lambdamb(\TAb{1v}-\TAb{2u})+\Lambdamb_v\TAb{1}-\Lambdamb_u\TAb{2}=(\GAb{1v}-\GAb{2u})\Lambdamb+\GAb{1}\GAb{2}\Lambdamb-\GAb{2}\GAb{1}\Lambdamb-\GAb{1}\Lambdamb\TAb{2}\nonumber\\
	+\GAb{2}\Lambdamb\TAb{1}\label{f3.1}
	\end{align}
\end{subequations}
Then,
\begin{subequations}
\begin{align}
	\Lambdamb(\TAb{1v}-\TAb{2u})+(\GAb{1}\Lambdamb-\Lambdamb_u)\TAb{2}+(\Lambdamb_v-\GAb{2}\Lambdamb)\TAb{1}=(\GAb{1v}-\GAb{2u})\Lambdamb+[\GAb{1},\GAb{2}]\Lambdamb\label{f3.2}
\end{align}
\end{subequations}
Using (\ref{f1}) and (\ref{f2}) on the left side
\begin{subequations}
	\begin{align}
	\Lambdamb(\TAb{1v}-\TAb{2u})-\Lambdamb\TAb{2}\TAb{1}+\Lambdamb\TAb{1}\TAb{2}=(\GAb{1v}-\GAb{2u})\Lambdamb+[\GAb{1},\GAb{2}]\Lambdamb\label{f3.2}
	\end{align}
\end{subequations}
Therefore, we have
\begin{subequations}
	\begin{align}
	\Lambdamb(\TAb{1v}-\TAb{2u}+[\TAb{1},\TAb{2}])=(\GAb{1v}-\GAb{2u}+[\GAb{1},\GAb{2}])\Lambdamb\label{f3.3}
	\end{align}
\end{subequations}
As $U-det^{-1}(\Lambdamb)(0)$ is dense in $U$ and $\Lambdamb$ is invertible there, we have the result. 
\end{proof}

\begin{lem}\label{lema2}
		If we have:
	\begin{subequations}
		\begin{align}
		\Ib=\Lambdamb\Ib_\Omega\Lambdamb^T\label{ff11}\\
		\GAb{1}\Lambdamb-\Lambdamb_u=\Lambdamb\TAb{1}\label{ff12}\\
		\GAb{2}\Lambdamb-\Lambdamb_v=\Lambdamb\TAb{2}\label{f22}
		\end{align}
	\end{subequations}
in which $\Ib,\Ib_\Omega,\Lambdamb,\TAb{1},\TAb{2}:U \to \mathcal{M}_{n\times n}(\R)$ and $\GAb{1},\GAb{2}:U-det^{-1}(\Lambdamb)(0) \to \mathcal{M}_{n\times n}(\R)$ are smooth maps with $int(det^{-1}(\Lambdamb)(0))=\emptyset$ and $det(\Ib_\Omega)\neq0$. Then,
\begin{itemize}
\item $\Ib\GAb{1}^T+\GAb{1}\Ib=\Ib_u$ if and only if, $\Ib_\Omega\TAb{1}^T+\TAb{1}\Ib_\Omega=\Ib_{\Omega u}$ on $U$.
\item $\Ib\GAb{2}^T+\GAb{2}\Ib=\Ib_v$ if and only if, $\Ib_\Omega\TAb{2}^T+\TAb{2}\Ib_\Omega=\Ib_{\Omega v}$ on $U$.	
\end{itemize}
\end{lem}
\begin{proof}
	The proof of the second item is analogous to the first one, so we are going to prove just the first. For $\p\in U-det^{-1}(\Lambdamb)(0)$, by (\ref{ff12}) we have,
	\begin{align}\label{tg}
	\Ib_\Omega\TAb{1}^T+\TAb{1}\Ib_\Omega&=\Ib_\Omega(\Lambdamb^T\GAb{1}^T-\Lambdamb^T_u)(\Lambdamb^T)^{-1}+\Lambdamb^{-1}(\GAb{1}\Lambdamb-\Lambdamb_u)\Ib_\Omega\\
	&=\Ib_\Omega\Lambdamb\GAb{1}^T(\Lambdamb^T)^{-1}+\Lambdamb^{-1}\GAb{1}\Lambdamb\Ib_\Omega-\Ib_\Omega\Lambdamb^T_u(\Lambdamb^T)^{-1}-\Lambdamb^{-1}\Lambdamb_u\Ib_\Omega\nonumber	
	\end{align}
	On the other hand, $\Lambdamb^{-1}\Lambdamb=\id_n$, then $\Lambdamb^T_u(\Lambdamb^T)^{-1}=-\Lambdamb^T((\Lambdamb^T)^{-1})_u$, $\Lambdamb^{-1}\Lambdamb_u=-(\Lambdamb^{-1})_u\Lambdamb$. Also, from (\ref{ff11}) $\Ib_\Omega\Lambdamb^T=\Lambdamb^{-1}\I$, $\Lambdamb\Ib_\Omega=\Ib(\Lambdamb^T)^{-1}$ substituting the last four equalities in (\ref{tg}) we get:
	\begin{align}
	\Ib_\Omega\TAb{1}^T+\TAb{1}\Ib_\Omega&=\Lambdamb^{-1}\Ib\GAb{1}^T(\Lambdamb^T)^{-1}+\Lambdamb^{-1}\GAb{1}\Ib(\Lambdamb^T)^{-1}-\Ib_\Omega\Lambdamb^T_u(\Lambdamb^T)^{-1}-\Lambdamb^{-1}\Lambdamb_u\Ib_\Omega\nonumber\\
	&=\Lambdamb^{-1}(\Ib\GAb{1}^T+\GAb{1}\Ib)(\Lambdamb^T)^{-1}+\Ib_\Omega\Lambdamb^T((\Lambdamb^T)^{-1})_u+(\Lambdamb^{-1})_u\Lambdamb\Ib_\Omega\nonumber\\
	&=\Lambdamb^{-1}(\Ib\GAb{1}^T+\GAb{1}\Ib)(\Lambdamb^T)^{-1}+\Lambdamb^{-1}\Ib((\Lambdamb^T)^{-1})_u+(\Lambdamb^{-1})_u\Ib(\Lambdamb^T)^{-1}\nonumber	
	\end{align}
By hypothesis $\Ib\GAb{1}^T+\GAb{1}\Ib=\Ib_u$, then
\begin{align}
\Ib_\Omega\TAb{1}^T+\TAb{1}\Ib_\Omega&=\Lambdamb^{-1}\Ib_u(\Lambdamb^T)^{-1}+\Lambdamb^{-1}\Ib((\Lambdamb^T)^{-1})_u+(\Lambdamb^{-1})_u\Ib(\Lambdamb^T)^{-1}\nonumber\\
&=(\Lambdamb^{-1}\Ib(\Lambdamb^T)^{-1})_u=\Ib_{\Omega u}\nonumber
\end{align}
By density of $U-det^{-1}(\Lambdamb)(0)$, $\Ib_\Omega\TAb{1}^T+\TAb{1}\Ib_\Omega=\Ib_{\Omega u}$ holds on $U$. The converse is obtained in the same way.     
\end{proof}

\begin{proof}[Proof. Teorema \ref{TF}(Existence)]
	By proposition \ref{metric} there exist $\TA{1}, \TA{2}:U \to \mathcal{M}_{2\times2}(\R)$ smooth maps such that on $(\laO^{-1}(0))^c$,
	\begin{subequations}
		\begin{align}
		&\TA{1}=\Lambdam^{-1}(\GA{1}\Lambdam-\Lambdam_u),\label{CTFF1}\\  &\TA{2}=\Lambdam^{-1}(\GA{2}\Lambdam-\Lambdam_v).\label{CTFF2}
		\end{align}
	\end{subequations}
	 Let us construct $\TAb{1}$ and $\TAb{2}$ as the matrices $\P$ and $\Q$ in (\ref{P}) and (\ref{Q}) respectively, using (\ref{W}), (\ref{IO}) and (\ref{IIO}). By (\ref{CTFF1}), (\ref{CTFF2}) and since $\boldsymbol{\alpha}\Lambdam=\boldsymbol{\mu}$ on $(\laO^{-1}(0))^c$ (caused by (\ref{FFS1}) and (\ref{FFS2})) we have for all $(u,v)\in (\laO^{-1}(0))^c$,
	\begin{subequations}
	\begin{align}
	\GAb{1}\Lambdamb-\Lambdamb_u=\Lambdamb\TAb{1}\text{ and }\GAb{2}\Lambdamb-\Lambdamb_v=\Lambdamb\TAb{2}\nonumber
	\end{align}
	\end{subequations}
where,
\begin{align}
\GAb{1}&=\begin{pmatrix}
\Ga{1}{1}{1} & \Ga{1}{2}{1} & e \\	
\Ga{2}{1}{1} & \Ga{2}{2}{1} & f \\
\alpha_{11} & \alpha_{12} & 0
\end{pmatrix}=\begin{pmatrix}
\frac{1}{2}E_u & (F_u-\frac{1}{2}E_v) & e \\	
\frac{1}{2}E_v & \frac{1}{2}G_u & f \\
-e & -f & 0
\end{pmatrix}\begin{pmatrix}
E & F & 0 \\
F & G & 0 \\
0 & 0 & 1\end{pmatrix}^{-1}\label{Gab1}\\ 
\GAb{2}&=\begin{pmatrix}
\Ga{1}{1}{2} & \Ga{2}{1}{2} & \alpha_{21} \\
\Ga{1}{2}{2} & \Ga{2}{2}{2} & \alpha_{22} \\
f & g & 0\end{pmatrix}=\begin{pmatrix}
\frac{1}{2}E_v & \frac{1}{2}G_u & -f \\
(F_v-\frac{1}{2}G_u) & \frac{1}{2}G_v & -g \\
f & g & 0\end{pmatrix}\begin{pmatrix}
E & F & 0 \\
F & G & 0 \\
0 & 0 & 1\end{pmatrix}^{-1}\label{Gab2}\\
\Lambdamb&=\begin{pmatrix}
\lambda_{11} & \lambda_{12} & 0 \\
\lambda_{21} & \lambda_{22} & 0 \\
0 & 0 & 1\end{pmatrix}
\end{align} 
Let $(u_0,v_0)\in U$, $\mathbf{q}\in\R^3$ be fixed points and since $\EO\GO-\FO^2>0$ we can find $\mathbf{z}_1$, $\mathbf{z}_2$, $\mathbf{z}_3$ fixed vectors of $\R^3$ linearly independent and positively oriented such that $\mathbf{z}_1\cdot\mathbf{z}_1=\EO(u_0,v_0)$, $\mathbf{z}_1\cdot\mathbf{z}_2=\FO(u_0,v_0)$, $\mathbf{z}_2\cdot\mathbf{z}_2=\GO(u_0,v_0)$, $\mathbf{z}_3\cdot\mathbf{z}_3=1$ and $\mathbf{z}_3\cdot\mathbf{z}_i=0$ for $i=1, 2$. Consider the system of PDE,
	\begin{subequations}
		\begin{align}
		\Wm^T_u&=\TAb{1}\Wm^T\label{eq1}\\ 
		\Wm^T_v&=\TAb{2}\Wm^T\label{eq2}\\
		\Wm(u_0,v_0)&=\begin{pmatrix}\mathbf{z}_1&\mathbf{z}_2&\mathbf{z}_3\end{pmatrix}
		\end{align}
	\end{subequations}
It is known in classical differential geometry that, the Gauss and Mainardi-Codazzi equations are equivalent to $\GAb{1v}-\GAb{2u}+[\GAb{1},\GAb{2}]=0$, then as this is satisfied, by lemmma \ref{l1}  $\TAb{1v}-\TAb{2u}+[\TAb{1},\TAb{2}]=0$ in $U$ which is the compatibility condition of the above system of equations. By corollary \ref{frobl}, this system has a unique solution $\Wm:\bar{V}\to GL(3)$, where  $\bar{V}$ is a neighborhood of $(u_0,v_0)$. Since $det(\Wm(u_0,v_0))>0$, restricting $\bar{V}$ if it is necessary, we can suppose that $det(\Wm)>0$ on $\bar{V}$. Setting the matrices, 
\begin{align}
	\Ib&:=\begin{pmatrix}
	E & F & 0 \\
	F & G & 0 \\
	0 & 0 & 1\end{pmatrix}, \Ib_\Omega:=\begin{pmatrix}
	\EO & \FO & 0 \\
	\FO & \GO & 0 \\
	0 & 0 & 1\end{pmatrix}\nonumber
	\end{align}
\begin{align}
\mathds{Y}:=\Wm^T\Wm\label{XS}
\end{align}
We want to prove that $\Ib_\Omega=\mathds{Y}$. Consider the following system of PDE. 
\begin{subequations}
	\label{epf}
	\begin{align}
	\mathds{Y}_u&=\mathds{Y}\TAb{1}^T+\TAb{1}\mathds{Y}\\ 
	\mathds{Y}_v&=\mathds{Y}\TAb{2}^T+\TAb{2}\mathds{Y}\\
	\mathds{Y}(u_0,v_0)&=\Ib_\Omega(u_0,v_0)
	\end{align}
\end{subequations}
Defining $\mathbf{\Theta}(u,v,\mathbf{X}):=\mathbf{X}\TAb{1}^T+\TAb{1}\mathbf{X}$ and $\mathbf{\Xi}(u,v,\mathbf{X}):=\mathbf{X}\TAb{2}^T+\TAb{2}\mathbf{X}$ for $\mathbf{X}\in \mathcal{M}_{3\times 3}(\R)$, we can compute the compatibility condition \ref{compatibility1} and we get:
\begin{align}
	&\mathbf{X}\TAb{1v}^T+\TAb{1v}\mathbf{X}+(\mathbf{X}\TAb{2}^T+\TAb{2}\mathbf{X})\TAb{1}^T+\TAb{1}(\mathbf{X}\TAb{2}^T+\TAb{2}\mathbf{X})\\ 
	=&\mathbf{X}\TAb{2u}^T+\TAb{2u}\mathbf{X}+(\mathbf{X}\TAb{1}^T+\TAb{1}\mathbf{X})\TAb{2}^T+\TAb{2}(\mathbf{X}\TAb{1}^T+\TAb{1}\mathbf{X})\nonumber
\end{align}
Eliminating common terms and grouping we have:
\begin{align}
&\mathbf{X}(\TAb{1v}^T-\TAb{2u}^T)+(\TAb{1v}-\TAb{2u})\mathbf{X}\\ 
=&\mathbf{X}(\TAb{1}^T\TAb{2}^T-\TAb{2}^T\TAb{1}^T)+(\TAb{2}\TAb{1}-\TAb{1}\TAb{2})\mathbf{X}\nonumber
\end{align}
then,
\begin{align}
&\mathbf{X}(\TAb{1v}-\TAb{2u}+[\TAb{1},\TAb{2}])^T+(\TAb{1v}-\TAb{2u}+[\TAb{1},\TAb{2}])\mathbf{X}=0\label{comp}
\end{align}
As $\TAb{1v}-\TAb{2u}+[\TAb{1},\TAb{2}]=0$, (\ref{comp}) is satisfied for all $\mathbf{X}\in \mathcal{M}_{3\times 3}(\R)$, then by theorem \ref{frobg} the system of PDE (\ref{epf}) has unique solution. On the other hand, using (\ref{eq1}) and (\ref{eq2}), it can be verified easily that $\mathds{X}$ defined in \ref{XS} is a solution of the system \ref{epf}. Also by (\ref{Gab1}) and (\ref{Gab2}) we have $\Ib\GAb{1}^T+\GAb{1}\Ib=\Ib_u$ and $\Ib\GAb{2}^T+\GAb{2}\Ib=\Ib_v$ on $(\lambda^{-1}(0))^c$, then by lemma \ref{lema2}, $\Ib_\Omega\TAb{1}^T+\TAb{1}\Ib_\Omega=\Ib_{\Omega u}$ and $\Ib_\Omega\TAb{2}^T+\TAb{2}\Ib_\Omega=\Ib_{\Omega v}$ on $U$, it means, $\Ib_\Omega$ is also a solution of the system \ref{epf}, therefore by uniqueness $\Ib_\Omega=\mathds{Y}$ on any neighborhood $\hat{V}$ of $(u_0,v_0)$. Now, as $\Ib_\Omega=\Wm^T\Wm$, we have that $\w_{3}$ is orthogonal to $\w_{1}$,$\w_{2}$ and $\w_{3}\cdot \w_{3}=1$. Since $det(\Wm)>0$, $\n:=\frac{\w_1\times\w_2}{\|\w_1\times\w_2\|}=\w_{3}$ and if we define $\Omegam:=\begin{pmatrix}\w_1&\w_2\end{pmatrix}$ then, $$\Omegam^T\Omegam=\begin{pmatrix}
\EO&\FO\\
\FO&\GO
\end{pmatrix}=\I_\Omega$$
from (\ref{eq1}) and (\ref{eq2}) we have,
\begin{subequations}
	\begin{align}
	\begin{pmatrix}
	\Ta{1}{1}{1} & \Ta{1}{2}{1} & \eo \\	
	\Ta{2}{1}{1} & \Ta{2}{2}{1} & \fdo \\
	\mu_{11} & \mu_{12} & 0
	\end{pmatrix}=\Wm^T_u\Wm\Ib_\Omega^{-1}&=\begin{pmatrix}\Omegam_u^T\Omegam\I_\Omega^{-1}&\Omegam_u^T\n\\
	\n_u^T\Omegam\I_\Omega^{-1}&0 \end{pmatrix}\\ 
	\begin{pmatrix}
	\Ta{1}{1}{2} & \Ta{1}{2}{2} & \fuo \\
	\Ta{2}{1}{2} & \Ta{2}{2}{2} & \go \\
	\mu_{21}     &   \mu_{22}   & 0\end{pmatrix}=\Wm^T_v\Wm\Ib_\Omega^{-1}&=\begin{pmatrix}\Omegam_v^T\Omegam\I_\Omega^{-1}&\Omegam_v^T\n\\
	\n_v^T\Omegam\I_\Omega^{-1}&0 \end{pmatrix}
	\end{align}
\end{subequations}
then,
\begin{subequations}
	\begin{align}
	&\TA{1}=\begin{pmatrix}
	\Ta{1}{1}{1} & \Ta{1}{2}{1}\\
	\Ta{2}{1}{1} & \Ta{2}{2}{1}
	\end{pmatrix}=(\Omegam_u^T\Omegam)\I_\Omega^{-1}\text{ and }\TA{2}=\begin{pmatrix}
	\Ta{1}{1}{2} & \Ta{1}{2}{2}\\
	\Ta{2}{1}{2} & \Ta{2}{2}{2}
	\end{pmatrix}=(\Omegam_v^T\Omegam)\I_\Omega^{-1}\nonumber
	\\
	&\II_\Omega=\begin{pmatrix}
\n\cdot\w_{1u} & \n\cdot\w_{1v}\\
\n\cdot\w_{2u} & \n\cdot\w_{2v}
\end{pmatrix}=\begin{pmatrix}
\eo&\fuo\\
\fdo&\go
\end{pmatrix}\nonumber 
\end{align}
\end{subequations}
Let us consider the system of PDE restricted to $\hat{V}$,
\begin{subequations}
\label{exist}
\begin{align}
&\x_u=\lambda_{11}\w_1+\lambda_{12}\w_2\\
&\x_v=\lambda_{21}\w_1+\lambda_{22}\w_2\\
&\x(u_0,v_0)=\mathbf{q}
\end{align}
\end{subequations}
As,
$$\begin{pmatrix}
0&1
\end{pmatrix}(\Lambdam\TA{1}+\Lambdam_u)=\begin{pmatrix}
0&1
\end{pmatrix}\GA{1}\Lambdam=\begin{pmatrix}
1&0
\end{pmatrix}\GA{2}\Lambdam=\begin{pmatrix}
1&0
\end{pmatrix}(\Lambdam\TA{2}+\Lambdam_v)$$
for $(u,v)\in(\lambda^{-1}(0))^c$, then by density $$\begin{pmatrix}
0&1
\end{pmatrix}(\Lambdam\TA{1}+\Lambdam_u)=\begin{pmatrix}
1&0
\end{pmatrix}(\Lambdam\TA{2}+\Lambdam_v)$$ on the entire $U$, as also, by (\ref{FFS2}) $\Lambdam\II_\Omega$ is symmetric, then the singular compatibility equations (\ref{CS1}), (\ref{CS2}) and (\ref{CS3}) are satisfied, which are the compatibility condition of the system (\ref{exist}). Therefore by theorem \ref{frobg}, this system has a solution $\x:V \to \x(V)\subset \R^3$, where $V\subset\hat{V}$ is a neighborhood of $(u_0,v_0)$. As $D\x=\Omegam\Lambdam^T$, by proposition \ref{OD}, $\x$ is a frontal with $\Omegam$ being a tangent moving base of it, satisfying what we wished.            
\end{proof}
\begin{proof}[Proof. Teorema \ref{TF}(Rigidity)]
Let ${\mathbf{\bar{\x}}}:U \to {\mathbf{\bar{\x}}}(U)\subset \R^3$ be a frontal, $U$ connected, with $\bar{\Omegam}$ a tangent moving base of $\bar{\x}$ satisfying the same conditions of $\x$ and $\Omegam$. As $\I_\Omega=\I_{\bar{\Omega}}$, exists a rotation $\rhom \in SO(3)$ such that $\rhom\Omegam(u_0,v_0)=\bar{\Omegam}(u_0,v_0)$. Set $\mathbf{a}:=\bar{\x}(u_0,v_0)-\rhom\x(u_0,v_0)$, $\hat{\x}:=\rhom\x+\mathbf{a}$ and $\hat{\Omegam}:=\rhom\Omegam$. Observe that, $\bar{\x}(u_0,v_0)=\hat{\x}(u_0,v_0)$, $\hat{\Omegam}(u_0,v_0)=\bar{\Omegam}(u_0,v_0)$, $D\hat{\x}=\hat{\Omegam}\Lambdam^T$,  $\I_\Omega=\I_{\hat{\Omega}}$ and $\II_\Omega=\II_{\hat{\Omega}}$ (caused by $\rhom\w_{1}\times\rhom\w_{2}=\rhom(\w_1\times\w_2)$). Also by remark \ref{rem1} $\TA{i}=\TAb{i}=\TAh{i}$. We want to prove that $\bar{\x}=\hat{\x}$ on $U$, so, let us define the set,$$\mathcal{B}:=\{(u,v)\in U:\bar{\Omegam}(u,v)=\hat{\Omegam}(u,v)\}$$
$\mathcal{B}$ is not empty and closed by continuity. For each $(\bar{u},\bar{v})\in\mathcal{B}$, as we saw in section 4, $\begin{pmatrix}\bar{\Omegam}&\bar{\n} \end{pmatrix}$ is a solution of the system:
\begin{subequations}
\begin{align}
&\Wm^T_u=\P\Wm^T\\ 
&\Wm^T_v=\Q\Wm^T\\
&\Wm(\bar{u},\bar{v})=\begin{pmatrix}\bar{\Omegam}(\bar{u},\bar{v})&\bar{\n}(\bar{u},\bar{v}) \end{pmatrix}
\end{align}
\end{subequations}
As the matrices $\P$ (\ref{P}) and $\Q$ (\ref{Q}) are constructed with the coefficients of $\I_\Omega$, $\II_\Omega$ and $\TA{i}$, then $\begin{pmatrix}\hat{\Omegam}&\hat{\n} \end{pmatrix}$ is solution of the system as well and by uniqueness, $\hat{\Omegam}=\bar{\Omegam}$ on a neighborhood of $(\bar{u},\bar{v})$. We have that $\mathcal{B}$ is open and since $U$ is connected, $\mathcal{B}=U$. Therefore, $D\bar{\x}=\bar{\Omegam}\Lambdam^T=\hat{\Omegam}\Lambdam^T=D\hat{\x}$ and since $\bar{\x}(u_0,v_0)=\hat{\x}(u_0,v_0)$, $\bar{\x}=\hat{\x}$ on $U$.       
\end{proof}
\begin{rem}
In theorem \ref{TF} can be switched the hypothesis of $E,F,G,e,f,g$ satisfying the Gauss and Mainardi-Codazzi equations for all $(u,v)\in U-\laO^{-1}(0)$ by $\EO,\FO,\GO,\eo,\fuo,\fdo,\go$ satisfying the equations (\ref{GaussT}), (\ref{c7}) and (\ref{c8}) on $U$, where $\TA{1},\TA{2}$ are defined as in proposition \ref{metric} (see remark \ref{tau}). Since $\TAb{1v}-\TAb{2u}+[\TAb{1},\TAb{2}]=0$ is equivalent to (\ref{GaussT}), (\ref{c7}) and (\ref{c8}), using lemma \ref{l1} these two different hypothesis are equivalent, then we obtain the same result in the theorem. By last, the frontal obtained is going to be a wave front if $(\KO,\HO)\neq(0,0)$ on the domain, where $\KO,\HO$ are computed with the given coefficients $\EO,\FO,\GO,\eo,\fuo,\fdo,\go$ and $\lambda_{ij}$.      
\end{rem}	

\def\cprime{$'$}

\end{document}